\documentclass{amsart}

\usepackage{amsmath,amssymb,amsfonts,graphics,xcolor}
\usepackage{epsfig}
\usepackage{latexsym}
\usepackage{euscript}
\usepackage{float}

\title{Conditions on square geometric graphs}

\author[Chuangpishit et al.]{Huda Chuangpishit \and Jeannette Janssen}
\address{Department of Mathematics \& Statistics, Dalhousie University, Halifax, Nova Scotia, Canada, B3H 3J5}{

\date{\today}

\newcommand{\ignore}[1]{}
\newtheorem{theorem}{Theorem}[section]
\newtheorem{corollary}[theorem]{Corollary}

\newtheorem{definition}[theorem]{Definition}

\newtheorem{assumption}[theorem]{Assumption}

\newtheorem{lemma}[theorem]{Lemma}
\newtheorem{proposition}[theorem]{Proposition}

\newcommand{\cal}{\mathcal}
\newcommand{\Rrr}{{\mathbb R}}

\def\C{{\mathcal C}}

\def\B{{\mathcal B}}
\def\A{{\mathcal A}}

\begin{document}

\begin{abstract} 
For any metric $d$ on $\Rrr^2$, an ($\Rrr^2,d$)-geometric graph is a graph whose vertices are points in $\Rrr^2$, and two vertices are adjacent if and only if their distance is at most 1.
If $d=\|.\|_{\infty}$, the metric derived from the $L_{\infty}$ norm, then $(\Rrr ^2,\|.\|_{\infty})$-geometric graphs are precisely those graphs that are the intersection of two unit interval graphs.  
We refer to $(\Rrr^2,\|.\|_{\infty})$-geometric graphs as square geometric graphs. We represent a characterization of square geometric graphs. Using this characterization we provide necessary conditions for the class of square geometric $B_{a,b}$-graphs, a generalization of cobipartite graphs.  Then by applying some restrictions on these necessary conditions we obtain sufficient conditions for $B_{a,b}$-graphs to be square geometric. 
\end{abstract}

\maketitle

\section{Introduction}
A given graph $G$ can be represented in very different layouts. Different representations of a graph have broad applications in areas such as social network analysis, graph visualization, etc. 
  
The $n$-dimensional geometric representation of a graph is a representation in which, the vertices of the graph are embedded in $\mathbb{R}^n$ equipped with an arbitrary metric $d$, and two vertices are adjacent if and only if their distance is at most 1. A graph $G$ is an ($\Rrr^n,d$)-geometric graph, if it has an $n$-dimensional geometric representation. If $d=\|.\|_{\infty}$, the metric derived from the $L_{\infty}$ norm, then $(\Rrr ^n,\|.\|_{\infty})$-geometric graphs are precisely those graphs that are the intersection of $n$ unit interval graphs. For $x=(x_1, \ldots, x_n),$ and $y=(y_1, \ldots, y_n)$ the distance of
$x$ and $y$ in $d_{\infty}$ metric is $\|x-y\|_\infty={\mbox max}_i|x_i-y_i|$.

Another way to define $(\Rrr^n,\|.\|_{\infty})$-geometric graphs is to look at it as the problem of representing a graph as the intersection graph of \emph{$n$-cubes} where an $n-$cube is the cartesian product of $n$ closed intervals of unit length of real line $\Rrr$.    
The minimum dimension of the space $\mathbb{R}^n$ for which $G$ has an $n$-cube presentation is a graph parameter called the \emph{cubicity} of a graph.
The concept of cubicity of graphs was first introduced and studied by Roberts in \cite{roberts1969}. In his paper, \cite{roberts1969}, Roberts indicates that there is a tight connection between graphs with cubicity $k$ and unit interval graphs. 

\begin{theorem}[\cite{roberts1969} ]\label{thm:Robert-char}
The cubicity of a graph $G$ is $k$, where $k$ is a positive integer, if and only if $G$ is the intersection of $k$ unit interval graphs.
\end{theorem}

Earlier results on cubicity study the complexity of recognition of graphs with a certain cubicity. In his paper, \cite{yannakakis}, Yannakakis shows that recognition of graphs with cubicity $k$ is NP-hard for any $k\geq 3$. Later Brue in \cite{Heinz-Brue-1996} proves that the problem of recognition of graphs with cubicity 2 in general is an NP-hard problem. As for $(\Rrr,\|.\|_{\infty})$-geometric graphs or unit interval graphs, there are several results presenting linear time algorithms for recognition of unit interval graphs. See \cite{unit-interval-1, corneil1998}. 

One of the main directions in the study of $(\Rrr^n,\|.\|_{\infty})$-geometric graphs is investigating the cubicity of specific families of graphs. In \cite{upper-cub-2016}, the authors study graphs with low chromatic number. The cubicity of interval graphs has been studied in \cite{upper-cub-interval-2009}. The cubicity of threshold graphs, bipartite graphs, and hypercube graphs has been investigated in \cite{cub-threshold-2009, cub-vertex-cover-2009, cub-certaingraphs-2005, chandran-etal-2008}. More results on cubicity of graphs can be found in \cite{cub-asteroidal-free-2010,cub-product-2015, cub-certaingraphs-2005}.

There are several characterizations of unit interval graphs. Here we state a result which characterizes unit interval graphs based on their forbidden subgraphs. 
\begin{theorem}\label{thm:unit-int-char}
\cite{golumbic}
A graph $G$ is a unit interval graph if and only if $G$ is claw-free, chordal, and asteroidal triple-free.
\end{theorem} 

The class of graphs we will look into in this paper is the class of {\em binate interval graphs}.

\begin{definition}\label{def:binate-interval}
A {\emph binate interval graph} is a graph whose vertex set can be partitioned into two sets $U$ and $W$ such that the graphs induced by $U$ and $W$ are connected unit interval graphs. 
\end{definition}

We are interested in studying this class of graphs mainly because of its structure, that is two unit interval graphs and some edges between them. Therefore, the binate interval graphs can be seen as a model of interaction between two unit interval graphs. Since unit interval graphs have broad applications in practical problems studying binate interval graphs may find its application in future.
Our aim here is to take the first steps towards studying  $(\Rrr^2,\|.\|_{\infty})$-geometric binate interval graphs.
For the sake of simplicity, we refer to $(\Rrr^2,\|.\|_{\infty})$-geometric graphs as square geometric graphs.
In this paper, we study a subclass of binate interval graphs, called {\sl $\B_{a,b}$-graphs}. 
\begin{definition}\label{def:B{a,b}}
A $\B_{a,b}$-graph is a binate interval graph  whose vertex set can be partitioned into two sets  $X_a\cup X_b$ and $Y$, where $X_a$, $X_b$, $Y$ are cliques and $X_a\cap X_b\neq\emptyset$. 
\end{definition}
Since unit interval graphs have a natural representation as a sequence of cliques, studying this special class, ${\cal B}_{a,b}$-graph, will definitely provides some insight into the problem of recognition of square geometric binate interval graphs.
Our approach to study square geometric graphs is inspired by the following characterization of unit interval graphs. 
\begin{theorem}\label{thm:conseq-1s-property}
\cite{olariu-looges}
A graph $G$ is a unit interval graph if and only if there is an ordering $<$ on the vertex set of G such that for any $u,v,z\in V(G)$ we have 
\begin{equation*}\label{eq:uiq-order}
u<z<v\quad \mbox{and\quad} u\sim v \Rightarrow u\sim z \quad\mbox{and}\quad v\sim z
\end{equation*} 
\end{theorem}
In Section \ref{chapter3:Square Geometric-B-a,b}, we present a characterization of square geometric graphs based on the existence of two linear orders on the vertex set of $G$. Then in Section \ref{sec:nec-codistion} we use this ordering characterization of square geometric graphs to provide necessary conditions for square geometric $B_{a,b}$-graphs. These necessary conditions may not be sufficient conditions. But in Section \ref{sec:chapterp3-sufficiency}, by adding some restrictions to the necessary conditions we obtain sufficient conditions for a $B_{a,b}$-graph to be square geometric.  

\section{Square Geometric ${\cal B}_{a,b}-$graphs}
\label{chapter3:Square Geometric-B-a,b}

In this section, we first present a characterization of square geometric graphs. Then using this characterization we investigate the properties of square geometric $B_{a,b}$-graphs. 

\subsection{A characterization of square geometric graphs} 
Motivated by the characterization of unit interval graphs introduced in Theorem \ref{thm:unit-int-char}, we present the following characterization of square geometric graphs.
\begin{theorem}\label{thm:dim-2-cond-geom}
A graph $G$ is a square geometric graph if and only if there exist two linear orderings $<_1$ and $<_2$ on the vertex set of $G$ such that for every $u, v, x, y, a, b \in V(G)$, 
\begin{eqnarray}\label{eq:2-dim-geo}
\begin{array}{lllll}
u<_1 a <_1 v &\mbox{and}& u<_1 b <_1 v, &\mbox{and}& \quad u\sim v\\
x<_2 a <_2 y &\mbox{and}& x<_2 b <_2 y, &\mbox{and}& \quad x\sim y
\end{array}
\Rightarrow a\sim b
\end{eqnarray}
\end{theorem}
\begin{proof}
Suppose that $G$ is a square geometric graph. By definition, there exists an embedding of $G$ in $\mathbb{R}^2$ such that two vertices $u, v$ of $G$ are adjacent if and only if $\|{u-v}\|_\infty\leq 1$. Define $<_i$, $1\leq i\leq 2$, to be the ordering of vertices based on the increasing order of their coordinates in the $i$-th dimension, respectively. It is clear that $<_i$, $1\leq i\leq 2$, satisfy the condition mentioned in the statement of the theorem. More precisely, let $a,b\in V(G)$, and $a=(a_1,a_2)$ and $b=(b_1,b_2)$. Suppose that we have the following:
\begin{eqnarray}\label{eq:2}
\begin{array}{lllll}
u<_1 a <_1 v &\mbox{and}& u<_1 b <_1 v, &\mbox{and}& \quad u\sim v\\
x<_2 a <_2 y &\mbox{and}& x<_2 b <_2 y, &\mbox{and}& \quad x\sim y
\end{array}
\end{eqnarray}
We prove that for all $1\leq i\leq 2$, we have that $|a_i-b_i|\leq 1$, which implies that $\|a-b\|_{\infty}\leq 1$, and thus $a\sim b$.
Let $i=1$, and suppose that $u$ and $v$ are the vertices corresponding to $<_1$ in Equation \ref{eq:2}.
%
%
Let $u=(u_1,u_2)$ and $v=(v_1,v_2)$. 
Since $u\sim v$, we have that $\|u-v\|_{\infty}\leq 1$, and thus $|u_i-v_i|\leq 1$ for all $1\leq i\leq 2$. By definition of $<_1$, we have that $u_1< a_1 < v_1$ and $u_1< b_1< v_1$ in the first dimension. This implies that $|a_1-b_1|\leq 1$. A similar discussion for $i=2$ shows that $|a_2-b_2|\leq 1$, and thus $\|a-b\|_{\infty}\leq 1$. So $a\sim b$.

Now suppose that $G$ is a graph with linear orderings $<_i$, $1\leq i\leq 2$, which satisfy Equation \ref{eq:2-dim-geo}.
For all $<_i$, $1\leq i\leq 2$, we construct a corresponding set $E_i$ as follows. If $v_r, v_s\in V(G) $ such that $v_r\sim v_s$ and $v_r<_i v_s$, then for any $v_t\in V(G)$ such that $v_r<_i v_t<_i v_s$, we add edges $v_tv_r$ and $v_tv_s$ to $E_i$.

Now define $G_i$, $1\leq i\leq 2$, to be the graph with vertex set $V(G_i)=V(G)$, and edge set $E(G_i)=E(G)\cup E_i$.
For all $1\leq i\leq 2$, the linear order $<_i$ on vertices $V(G_i)$ satisfies Equation \ref{eq:uiq-order}. Then, by Theorem \ref{thm:conseq-1s-property}, we have that, for all $1\leq i\leq 2$, the graph $G_i$ is a unit interval graph. 
Now suppose that $ab$ is an edge in $E_1\cup E_2$. This implies that for all $1\leq i\leq 2$ there exist vertices $u, v\in V(G)$ such that $u$ and $v$ are adjacent, and moreover $u<_i a<_i v$ and $u<_i b<_i v$. Since linear orderings $<_i$, $1\leq i\leq 2$, satisfy Equation \ref{eq:2-dim-geo} then $ab\in E(G)$.
This implies that $E(G_1)\cap E(G_2)= E(G)$. Therefore, $G=G_1\cap G_2$. Since all $G_i$, $1\leq i\leq 2$, are unit interval graphs, then by Theorem \ref{thm:Robert-char} we have that $G$ is square geometric.
\end{proof}
Given the two orderings $<_1$ and $<_2$ as in Definition \ref{def:completion}, how can we say if they satisfy Equation \ref{eq:2-dim-geo}? In what follows, we will address this question.
 
\begin{definition}\label{def:completion}
Let $G$ be a square geometric graph with linear orders $<_1$ and $<_2$ as in Theorem \ref{thm:dim-2-cond-geom}. Define
$$E_i=\{wz |\exists u,v\in V(G),  u<_i w<_i v  \quad\mbox{and }\quad  u<_i z<_i v \quad\mbox{and }\quad u\sim v\}.$$
The \emph{completion} of $<_i$, denoted by $\C_i$, is $(E(G))^c\cap E_i$. Indeed $\C_i$ is the set of the non-edges of $G$ whose ends are in between two adjacent vertices in $<_i$. 
\end{definition}
Note that the completions $\C_1$ and $\C_2$ of Definition \ref{def:completion} are subsets of the set of non-edges of $G$. 
The following lemma shows the relation between linear orders of Theorem \ref{thm:dim-2-cond-geom}, $<_1$ and $<_2$, and the completions of Definition \ref{def:completion}, $\C_1$ and $\C_2$.
%
\begin{lemma}\label{lem:emptycompletion-general}
Let $G$ be a square geometric graph, and let $<_1$ and $<_2$ be linear orders on the vertex set of $G$. Then $<_1$ and $<_2$ satisfy Equation (\ref{eq:2-dim-geo}) if and only if $\C_1$ and $\C_2$, the completions of  $<_1$ and $<_2$ respectively, have empty intersection.
\end{lemma}
\begin{proof}
First suppose $<_1$ and $<_2$ satisfy Equation (\ref{eq:2-dim-geo}). By contradiction suppose $w,z\in V(G)$, and $wz\in \C_1\cap \C_2$. Then by Definition \ref{def:completion}, there are $u, v, x, y\in V(G)$ such that 
\begin{eqnarray}
\begin{array}{lllll}
u<_1 w <_1 v &\mbox{and}& u<_1 z <_1 v,&\mbox{and}& \quad u\sim v\\
x<_2 w <_2 y &\mbox{and}& x<_2 z <_2 y,&\mbox{and}& \quad x\sim y
\end{array}
\end{eqnarray}

Since $<_1$ and $<_2$ satisfy Equation (\ref{eq:2-dim-geo}), we have $w\sim z$. This contradicts the fact that $\C_1$ and $\C_2$ are subsets of non-edges of $G$. Therefore, $\C_1\cap \C_2=\emptyset$. 

Now suppose that $<_1$ and $<_2$ do not satisfy Equation (\ref{eq:2-dim-geo}). This implies that there are $u, v, x, y, w, z\in V(G)$ such that $w\nsim z$, and
\begin{eqnarray}
\begin{array}{lllll}
u<_1 w <_1 v &\mbox{and}& u<_1 z <_1 v,&\mbox{and}& \quad u\sim v\\
x<_2 w <_2 y &\mbox{and}& x<_2 z <_2 y,&\mbox{and}& \quad x\sim y
\end{array}
\end{eqnarray}

Then by the definition of completions (Definition \ref{def:completion}) we have that $wz\in\C_1$ and $wz\in\C_2$.  
 Therefore, $\C_1\cap\C_2\neq\emptyset$.

\end{proof}

\subsection{Square geometric $B_{a,b}$-graphs}
\label{subsec:Square-geo-B-a,b}
We now collect the properties of a square geometric $B_{a,b}$-graph. Let us start with the following definition.

\begin{definition}\label{def:rigid-pair}
Let $G$ be a $B_{a,b}$-graph and $x_1y_1,x_2y_2$ are two edges of $G$ with $x_1, x_2\in Xa\cup X_b$ and $y_1, y_2\in Y$. Then $\{x_1y_1,x_2y_2\}$ is called a {\sl rigid pair} of $G$ if $x_1y_1y_2x_2$ is an induced $4$-cycle of $G$. Moreover the non-edges $x_1y_2$ and $x_2y_1$ are called the {\sl chords} of the rigid pair $\{x_1y_1,x_2y_2\}$.
\end{definition}
Note that by Theorem \ref{thm:conseq-1s-property} we know that an induced 4-cycle is a forbidden subgraph of a unit interval graph. 
\begin{proposition}\label{prop:completion-empty}
Let $G$ be a square geometric $B_{a,b}$-graph with linear orders $<_1, <_2$ as in Equation (\ref{eq:2-dim-geo}). Then every completion  $\C_i$, $i\in\{1,2\}$ contains exactly one chord of any rigid pair.
\end{proposition}
\begin{proof}
Suppose $G$ is a square geometric $B_{a,b}$-graph with linear orders $<_1, <_2$ satisfying Equation (\ref{eq:2-dim-geo}). By Lemma \ref{lem:emptycompletion-general}, we know that $\C_1\cap \C_2=\emptyset$. This implies that a chord of a rigid pair belongs to at most one of the completions $\C_i$, $i=1,2$.
We now show that a chord of a rigid pair belongs to either $\C_1$ or $\C_2$.
Let $\{x_1y_1,x_2y_2\}$ be a rigid pair of $G$. Without loss of generality let $x_1<_1 x_2$. Using the fact that Equation (\ref{eq:2-dim-geo}) holds for $<_1$ and $<_2$, and $x_1\sim y_1$ we have:
\begin{itemize}
\item If $x_1<_1 y_1$, then either $x_1<_1 y_1 <_1 x_2$ or $x_1<_1 x_2 <_1 y_1$. Thus $x_2y_1\in \C_1$.
\item If $y_2<_1 x_2$, then either $y_2<_1 x_1 <_1 x_2$ or $x_1<_1 y_2 <_1 x_2$. Thus $x_1y_2\in \C_1$.
\item If neither $x_1<_1 y_1$ nor $y_2<_1 x_2$, then we have $y_1<_1 x_1<_1 x_2<_1 y_2$.  This implies that $x_1y_2\in \C_1$ and $x_2y_1\in \C_1$.
\end{itemize}
Therefore $\C_1$ includes at least one chord of $\{x_1y_1,x_2y_2\}$. A similar discussion for $\C_2$ proves that $\C_2$ includes at least one chord of $\{x_1y_1,x_2y_2\}$. Note that since $\C_1$ and $\C_2$ contains  at least one chord of $\{x_1y_1,x_2y_2\}$ and $\C_1\cap\C_2=\emptyset$ then the third case never occurs.
\end{proof}
We know by Proposition \ref{prop:completion-empty} that the completions $\C_1$ and $\C_2$ provide a bipartion of the non-edges of $G$ which are chords of some rigid pairs of $G$. To study these non-edges we define a graph associated with $G$. 
\begin{definition}\label{def:chord-graph}
Let $G$ be a $B_{a,b}$-graph with clique bipartition $X=X_a\cup X_b$ and $Y$. 
The {\sl chord graph} of $G$, denoted by $\tilde{G}$ is defined as follows. 
$$V(\tilde{G})= \{xy | x\in X, y\in Y, xy\not\in E\}.$$
Two vertices of $\tilde{G}$ are adjacent if and only if they are the missing chords of an induced $4$-cycle of $G$, namely

$$E(\tilde{G})= \{uv | u=xy, v=x'y', xy'\hspace{.2cm}\mbox{and}\hspace{.2cm} x'y \hspace{.2cm} \mbox{are in }E\}. $$
\end{definition}
The vertex set of the chord graph, $\tilde{G}$, as in Definition \ref{def:chord-graph} is the set of non-edges of $G$. From now on we may use a ``vertex of $\tilde{G}$'' and  a ``non-edge of $G$'' interchangeably.  For clarity, we denote the adjacency in graph $\tilde{G}$ by $\sim ^*$. 
Note that by Definition \ref{def:chord-graph}, two vertices of $\tilde{G}$ are adjacent if and only if they are chords of a rigid pair. Therefore a non-edge of $G$ is either a chord of a rigid pair or an isolated vertex of $\tilde{G}$. Proposition \ref{prop:completion-empty} shows that the set of non-isolated vertices of $\tilde{G}$ is a subset of $\C_1\cup\C_2$, and two adjacent vertices of $\tilde{G}$ belong to different completions $\C_1$ and $\C_2$. This provides us with a bipartition  for the set of non-isolated vertices of $\tilde{G}$. The following corollary is an immediate consequence of this bipartition which presents a necessary condition for a $B_{a,b}$-graph to be square geometric.
\begin{corollary}\label{cor:coB-necessary-cond}
Let $G$ be a square geometric $B_{a,b}$-graph. Then its chord graph $\tilde{G}$ is bipartite.
\end{corollary}
\section{Necessary Conditions}
\label{sec:nec-codistion}
We saw in Subsection \ref{subsec:Square-geo-B-a,b} that a necessary condition for a $B_{a,b}$-graph $G$ to be square geometric is that its chord graph $\tilde{G}$ is bipartite. In this section, we will present more necessary conditions targeting the structure of the graph $G$ as well as its chord graph. 

Let $G$ be a $B_{a,b}$-graph with clique bipartition $X, Y$, where $X=X_a\cup X_b$, and $X_a$, $X_b$ and $Y$ are cliques. A vertex $v\in\tilde{G}$ is called an {\sl $a$-vertex} if there there exists $y\in Y$ such that $v=ay$ for $a\in X_a\setminus X_b$. For 
$b\in X_b\setminus X_a$, a {\sl $b$-vertex} is defined similarly.

\begin{assumption}\label{ass:1}
Let $G$ be a $\B_{a,b}$-graph with a connected chord graph, and clique bipartition $X_a\cup X_b$ and $Y$. Let $X_a\setminus X_b=\{a_1,\ldots, a_s\}$ and $X_b\setminus X_a=\{b_1,\ldots, b_r\}$.
Suppose that for all $v$ in $X_a\setminus X_b$ and $X_b\setminus X_a$ we have $0<|N_Y(v)|<|Y|$. Also, for all $b\in X_b\setminus X_a$, suppose that there exists $u\in X_a\cup X_b$ such that $N_Y(b)\not\subseteq N_Y(u)$. We assume $E(\tilde{G})$ has an edge such that none of its ends is an $a$-vertex or a $b$-vertex.
\end{assumption}

The reason we study the $B_{a,b}$-graphs whose chords graphs are connected, is that part of our methods are based on specific properties of some proper 2-colorings of the chord graphs. A disconnected chord graph $\tilde{G}$ has $2^c$ possible colorings, where $c$ is the umber of components of $\tilde{G}$. The process of searching among $2^c$ possible colorings to obtain the coloring which satisfies the required properties is challenging and needs more complicated discussions.     

Moreover, in Assumption \ref{ass:1}, we exclude some particular structures of a $\B_{a,b}$-graph $G$. 
All of the excluded cases of Assumption \ref{ass:1} have a simple-structured chord graph $\tilde{G}$ which makes dealing with these cases easier. However, the proofs for the excluded cases are slightly different from the proofs of the general $B_{a,b}$-graphs. So in this paper we focus on the general cases of Assumption \ref{ass:1}.

Throughout this section we assume that $G$ is a square geometric $B_{a,b}$-graph.
We assume that there are linear orders $<_1$ and $<_2$ for a $B_{a,b}$-graph $G$ which satisfy Equation (\ref{eq:2-dim-geo}) {\sl i.e.} $\C_1\cap\C_2=\emptyset$. This implies that the non-edges of $G$ belong to at most one completion $\C_1$ and $\C_2$. We use this fact to collect some necessary conditions on the structure of the graph $G$. The non-edges of a $B_{a,b}$-graph can be partitioned into the following classes: (1) The isolated vertices of $\tilde{G}$, (2) The non-isolated vertices of $\tilde{G}$, and (3) The non-edges of form $ab$ where $a\in X_a\setminus X_b$ and $b\in X_b\setminus X_a$. 

The isolated vertices of $\tilde{G}$ force no restriction on the structure of the graph $G$ as they are the non-edges that can be dealt with when defining the linear orders $<_1$ and $<_2$. 
We already saw in Corollary \ref{cor:coB-necessary-cond} that the non-edges of class (2) or the non-isolated vertices of $\tilde{G}$ force the chord graph to be bipartite. So for the rest of this section we will study the restrictions caused by the non-edges of class (3). 
There are some specific structures of the neighborhoods of vertices $a\in X_a\setminus X_b$ and $b\in X_b\setminus X_a$ which force the non-edges of part (3) to belong to both completions $\C_1$ and $\C_2$ for any two linear orders $<_1$ and $<_2$. In what follows, we introduce such forbidden structures called {\sl rigid-free conditions}. 
\begin{definition}\label{def:rigid-free}
Let $G$ be a $\B_{a,b}$-graph with bipartition $X_a\cup X_b$ and $Y$. Assume $v\in X_a\cup X_b$ and $S\subseteq X_a\cup X_b$. Then $v$ is called \emph{rigid-free with respect to $S$} if there is no rigid pair $\{x_1y_1,x_2y_2\}$ of $G$ with $\{x_1, x_2\}\subseteq S$ and $y_1, y_2\in N_Y(v)$. 
\end{definition}

\begin{definition}\label{def:rigid-free-conditions}
[Rigid-free conditions]
Let $G$ be as in Assumption \ref{ass:1}. Then the {\sl rigid-free conditions} are as follow. Either of the following statements is true.
\begin{itemize}
\item[(i)] For all $a\in X_a$, $a$ is rigid-free with respect to the sets $\{x_1,x_2\}$ and $\{x,b\}$, where $x, x_1, x_2\in X_a\cap X_b$ and $b\in X_b\setminus X_a$.
\item[(ii)] For all $b\in X_b\setminus X_a$, $b$ is rigid-free with respect to the sets $\{x_1,x_2\}$ and $\{x,a\}$, where $x, x_1, x_2\in X_a\cap X_b$ and $a\in X_a\setminus X_b$.

\end{itemize} 
\end{definition}
We will prove, in Subsection \ref{subsec:chap3:Nec-rigid-free}, that if a graph $G$, as given in Assumption \ref{ass:1}, is square geometric, then the rigid-free condition of Definition \ref{def:rigid-free-conditions} hold. 

Besides the rigid free conditions there is a coloring condition which must be satisfied if $G$ is a $B_{a,b}$-graph as in Assumption \ref{ass:1} and is square geometric. In sequel, we give an insight into this necessary coloring condition. 
%
%
%
%
Let $\mathcal{A}$ be the set of all non-isolated $a$-vertices which have a neighbor in $V(\tilde{G})$ that is not an $a$-vertex. Similarly, let $\mathcal{B}$ be the set of all $b$-vertices which have a neighbor in $V(\tilde{G})$ that is not a $b$-vertex. 
%

A necessary condition for a $B_{a,b}$-graph $G$ to be square geometric is that there exists a 2-coloring of $\tilde{G}$ such that $\mathcal{A}$ is a subset of one color class, and $\mathcal{B}$ is the subset of the other color class. 

The following theorem presents necessary conditions for a graph $G$ as in Assumption \ref{ass:1} to be square geometric.

\begin{theorem}\label{thm:chp3-main-result}
Let $G$ be a square geometric $B_{a,b}$-graph as given in Assumption \ref{ass:1}. Then the following conditions are satisfied:
\begin{itemize}
\item[(i)] There is a proper 2-coloring $f:V(\tilde{G})\rightarrow\{\mbox{red, blue}\}$ which colors all vertices of $\mathcal{A}$ red and all vertices of $\mathcal{B}$ blue.
\item[(ii)] The rigid-free conditions of Definition \ref{def:rigid-free-conditions} hold. 
\end{itemize} 
\end{theorem}

We will see, in Section \ref{sec:chp3:main-algorithm}, that the conditions of Theorem \ref{thm:chp3-main-result} can be checked in polynomial-time.
For the rest of this section, we assume that a graph $G$ is as in the following assumption.
\begin{assumption}\label{ass:nec-condition}
Let $G$ be a $B_{a,b}$-graph as given in Assumption \ref{ass:1}. Suppose $G$ is square geometric with linear orders $<_1$ and $<_2$ as in Equation (\ref{eq:2-dim-geo}). Let $\C_1$ and $\C_2$ be completions of $<_1$ and $<_2$, respectively. 
\end{assumption}
\subsection{Properties of completions $\C_1$ and $\C_2$}
\label{subsec:chp3-completion-properties}
In this subsection, we collect some properties of completions $\C_1$ and $\C_2$ of Assumption \ref{ass:nec-condition}. These properties will be used to prove Theorem \ref{thm:chp3-main-result}.
The following lemma is an easy consequence of Definition \ref{def:completion} (definition of completions). However the lemma is very useful, as the results of the lemma will be used to a great extent in future proofs. 
\begin{lemma}\label{lem:completion-properties}
Let $G$ be a square geometric graph with linear orders $<_1$ and $<_2$ and corresponding completions $\C_1$ and $\C_2$. Then the following statements hold for all $i\in\{1,2\}$. 
\begin{itemize}
\item[(1)] Let $u, v, w, z\in V(G)$ be such that $w\notin N(z)$ and $u, v\in N(z)$. If $u<_i w<_i v$ then $zw\in\C_i$.
\item[(2)] Let $u, v, w\in V(G)$ be such that $w\in N(v)\setminus N(u)$. If $u<_i v$ and $wu\notin\C_i$ then $u<_i w$. Similarly if $v<_i u$ and $wu\notin\C_i$ then $w<_i u$.
\item[(3)] Let $G$ be a $B_{a,b}$-graph. Suppose $x\in X_a\cup X_b$ and $y_1, y_2\in Y$. If $y_1<_i x<_i y_2$ then for all $y\in Y\setminus N_Y(x)$ we have $xy\in\C_i$. Similarly, let $y\in Y$ and $x_1, x_2\in X_a\cap X_b$. If $x_1<_i y<_i x_2$ then for all $x\in (X_a\cap X_b)\setminus N_X(y)$ we have $xy\in\C_i$.
\item[(4)] Let $u_1, u_2, w, z\in V(G)$ be such that $w\notin N(z)$, $u_1\in N(z)$ and $u_2\in N(w)$. If $u_1<_i w$ and $u_2<_i z$ then $zw\in\C_i$. Similarly, if $w<_i u_1$ and $z<_i u_2$ then $zw\in\C_i$.
\item[(5)] Let $u_1, u_2, v_1, v_2, w, z\in V(G)$ be such that $w\notin N(z)$, $u_1\in N(u_2)$ and $v_1\in N(v_2)$. If $u_1<_i w<_i v_2$ and $v_1<_i z<_i u_2$ then $zw\in\C_i$. 
\end{itemize}
\end{lemma}
%
%
%
%
%
3
Let $<_1$ and $<_2$ be linear orders of a square geometric graph. Suppose $S\subseteq V(G)$ and $v\in V(G)$. For any $i\in\{1,2\}$, we denote the statement \lq\lq for all $s\in S$ we have $v<_i s$ '' by $v<_i S$.
\begin{lemma}\label{lem:x_1<x_2iffy_1<y_2}
Let $G$ be a $B_{a,b}$-square geometric graph with linear orders $<_1$ and $<_2$, and corresponding completions $\C_1$ and $\C_2$. Let $\{xy,x'y'\}$ be a rigid pair of $G$ and $x'y\in \C_1$.
\begin{itemize}
\item[(1)] If $x<_1 x'$ then either $x<_1 x'<_1 y<_1 y'$ or $x<_1 y<_1 x'<_1 y'$. Moreover for any $u\in N(y')$ we have $x<_1 u$, and for any $v\in N(x)$ we have $v<_1 y'$. In particular, $x<_1 Y$ and $(X_a\cap X_b)<_1 y'$.
\item[(2)] If $x'<_1 x$ then either $y'<_1 y<_1 x'<_1 x$ or $y'<_1 x'<_1 y<_1 x$. Moreover for any $u\in N(y')$ we have $u<_1 x$, and for any $v\in N(x)$ we have $y'<_1 v$. In particular, $Y<_1 x$ and $y'<_1 (X_a\cap X_b)$.
\item[(3)] $x<_1 x'$ if and only if $y<_1 y'$.
\end{itemize}
If $x'y\in\C_2$ then statements (1)-(3) hold if we replace $<_1$ by $<_2$.
\end{lemma}
\begin{proof}
Let $<_1$ and $<_2$ be linear orders on $V(G)$ as in Equation \ref{eq:2-dim-geo}, with corresponding completions $\C_1$ and $\C_2$. By Proposition \ref{prop:completion-empty} we know that every chord of a rigid pair belongs to exactly one completion. Since $x'y\in\C_1$ then $x'y\notin\C_2$, and moreover $xy'\in\C_2\setminus\C_1$. We now prove (1). Assume $x<_1 x'$. We know $xy'\notin\C_1$, and $y'\in N_Y(x')\setminus N_Y(x)$. Then by part (2) of Lemma \ref{lem:completion-properties} we have $x<_1 x'<_1 y$. If $y<_1 x<_1 x'<_1 y'$ then by definition of completion and the fact that $y\sim y'$ we have $xy'\in\C_1$ which is not true. Similarly if $x<_1 x'<_1 y'<_1 y$ then by definition of completion and the fact that $x\sim y$ we have $xy'\in\C_1$ which is not true. This implies that either $x<_1 x'<_1 y<_1 y'$ or $x<_1 y<_1 x'<_1 y'$. 

Now suppose $u\in N(y')$. Then $y'\in N(u)\setminus N(x)$. If $u<_1 x<_1 y'$ then by part (2) of Lemma \ref{lem:completion-properties} we have $xy'\in\C_1$ which contradicts our assumption ($xy'\in\C_2$). Therefore $x<_1 u$. Since for all $u\in Y$ we have $u\in N(y')$ then $x<_1 Y$. Now let $v\in N(x)$. Then $x\in N(v)\setminus N(y')$. If $x<_1 y'<_1 v$ then by part (2) of Lemma \ref{lem:completion-properties} we have $xy'\in\C_1$ which contradicts our assumption ($xy'\in\C_2$). Therefore $v<_1 y'$. Since for all $v\in X_a\cap X_b$ we have $v\in N(x)$ then $(X_a\cap X_b)<_1 y'$. 

The proof of (2) is analogous. To prove part (3), suppose that $x<_1 x'$. Then, by part (1), we have that $y<_1 y'$. Moreover, if $y<_1 y'$ then we know that part (2) does not occur, and thus we have $x<_1 x'$. This finishes the proof of the lemma.
The proof for the case $x'y\in\C_2$ is analogous. Part (3) is an immediate result of (1) and (2).
\end{proof}
Lemma \ref{lem:x_1<x_2iffy_1<y_2} says that the embedding of a rigid pair $\{xy,x'y'\}$ in $(\Rrr^2,\|.\|_{\infty})$ has a general form as shown in Figure \ref{fig:emb-rigid-pair}. 
In Figure \ref{fig:emb-rigid-pair}, we assumed that the $x$-coordinates of the vertices give us the relation $<_1$ and the $y$-coordinates give the ordering $<_2$. 
\begin{figure}[H]
\label{fig:emb-rigid-pair}
\centerline
{\includegraphics[width=0.25\textwidth]{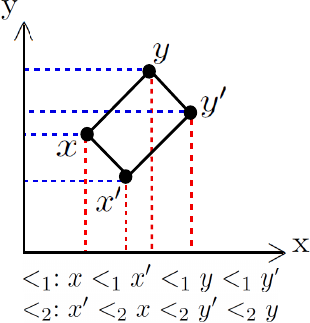}}
\caption{The embedding of a rigid pair $\{xy,x'y'\}$ in $(\Rrr^2,\|.\|_{\infty})$.}
\end{figure} 
\begin{lemma}\label{lem:x1<x2-iff-x'1<x'2}
Let $G$ be a $B_{a,b}$-square geometric graph with linear orders $<_1$ and $<_2$, and corresponding completions $\C_1$ and $\C_2$. Suppose $\{x_1y_1,x_2y_2\}$ and $\{x'_1y'_1,x'_2y'_2\}$ are rigid pairs of $G$ with $x_1, x'_1\in X_a$ or $x_1, x'_1\in X_b$. 
\begin{itemize}
\item[(1)] If $x_1y_2$ and $x'_1y'_2$ belong to the same completion then, for all $i\in\{1,2\}$, we have $x_1<_i x_2$ if and only if $x'_1<_i x'_2$. 
\item[(2)] If $x_1y_2$ and $x'_1y'_2$ belong to different completions then, for all $i\in\{1,2\}$, we have $x_1<_i x_2$ if and only if $x'_2<_i x'_1$. 
\end{itemize}
\end{lemma}
\begin{proof}
We only prove the lemma for $X_a$. The proof for $X_b$ follows by symmetry of $X_a$ and $X_b$.
We prove (1) for $i=1$. The proof for $i=2$ is analogous. Assume $\{x_1y_1,x_2y_2\}$ and $\{x'_1y'_1,x'_2y'_2\}$ are rigid pairs of $G$. First note that if $x_1y_2$ and $x'_1y'_2$ belong to the same completion then $x_2y_1$ and $x'_2y'_1$ belong to the same completion as well.

Suppose without loss of generality that $x_2y_1\in\C_1$ and $x'_2y'_1\in\C_1$. Consequently, $x_1y_2\in \C_2$ and $x'_1y'_2\in\C_2$.
Let $x_1<_1 x_2$. Since $x_2y_1\in\C_1$ by part (1) of Lemma \ref{lem:x_1<x_2iffy_1<y_2} we have $x_1<_1 y_2$. Also since $y_2'\in N(y_2)$ then by part (1) of Lemma \ref{lem:x_1<x_2iffy_1<y_2} we have $x_1<_1 y_2'$. Similarly, since $x_1, x'_1\in X_a$ we know that, $x_1'\in N(x_1)$, and thus $x_1'<_1 y_2$. Consequently, if $y_2'<_1 x_1'$ then $y_2'<_1 x_1'<_1 y_2$ and by part (2) of Lemma \ref{lem:completion-properties} we have that $x_1'y_2'\in\C_1$, which contradicts $x'_1y'_2\in\C_2$. Therefore, $x_1'<_1 y_2'$, and thus by Lemma \ref{lem:x_1<x_2iffy_1<y_2} we know that $x_1'<_1 x_2'$. If $x_1'<_1 x_2'$ then an analogous discussion proves that $x_1<_1 x_2$.

We now prove (2). Suppose that $x_1y_2$ and $x'_1y'_2$ belong to different completions. Without loss of generality let $x_1'y_2'\in\C_1$ and $x_1y_2\in\C_2$. This implies that $x'_2y'_1\in\C_2$ and $x_2y_1\in\C_1$. Then we have that $x_1'y_2'$ and $x_2y_1$ belong to the same completion. Therefore by part (1) we have that $x_1'<_i x_2'$ if and only if $x_2<_i x_1$.
\end{proof}
In the next few lemmas, we assume that the $B_{a,b}$-graph $G$ is as in Assumption \ref{ass:nec-condition}. We collect some properties of the non-edges of class (3) {\sl i.e.} the non-edges of form $ab$, where $a\in X_a\setminus X_b$ and $b\in X_b\setminus X_a$.
The next two auxiliary lemmas (Lemmas \ref{lem:location-1-a,b-rigid-free-conditions} and \ref{lem:location-2-a,b-rigid-free-conditions}) give us some information about the relation of the vertices of $G$ in the linear orders $<_1$ and $<_2$. 
%
\begin{lemma}\label{lem:location-1-a,b-rigid-free-conditions}
Let $G$ be a square geometric $\B_{a,b}$-graph as in Assumption \ref{ass:nec-condition}. Suppose $ab\notin\C_1$.
\begin{itemize}
\item[(1)] Let $b<_1 a$. Then we have $b<_1 (X_a\cap X_b)<_1 a$. Moreover, $N_Y(b)<_1 a$ and $b<_1 N_Y(a)$.
\item[(2)] Let $a<_1 b$. Then we have $a<_1 (X_a\cap X_b)<_1 b$. Moreover, $N_Y(a)<_1 b$ and $a<_1 N_Y(b)$.
\end{itemize}
If $ab\notin\C_2$ then (1) and (2) hold if we replace $<_1$ by $<_2$.
\end{lemma}
\begin{proof}
We only prove (1). The proof of (2) follows by exchanging $a$ and $b$ in the discussion for the proof of (1). Let $b<_1 a$ and $x\in X_a\cap X_b$. Assume by contradiction that $b<_1 a<_1 x$ or $x<_1 b<_1 a$. Since $x\in N(b)$ and $x\in N(a)$, by the definition of completion, for both cases we have $ab\in\C_1$, which contradicts our assumption. Therefore, for all $x\in X_a\cap X_b$ we have $b<_1 x<_1 a$.
Now let $y\in N_Y(b)$. Then if $b<_1 a<_1 y$ then since $y\in N(b)$ we have that $ab\in\C_1$ which contradicts $ab\notin\C_1$. Therefore, for all $y\in N(b)$ we have $y<_1 a$. Similarly for all $y\in N(a)$ we have $b<_1 y$.
\end{proof}
\begin{lemma}\label{lem:location-2-a,b-rigid-free-conditions}
Let $G$ be a square geometric $\B_{a,b}$-graph as in Assumption \ref{ass:nec-condition}. Suppose $ab\notin\C_1$.
\begin{itemize}
\item[(1)] Let $b<_1 a$. Then either $b<_1 Y$ or $Y<_1 a$.
\item[(2)] Let $a<_1 b$. Then either $a<_1 Y$ or $Y<_1 b$.
\end{itemize}
If $ab\notin\C_2$ then (1) and (2) hold if we replace $<_1$ by $<_2$.
\end{lemma}
\begin{proof}
We only prove (1). The proof of (2) follows by exchanging $a$ and $b$ in the proof of (1). Let $b<_1 a$. Suppose to the contrary that there are $y_1, y_2\in Y$ such that $y_1<_1 b<_1 a<_1 y_2$. Then, since $y_1\in N(y_2)$ by definition of completion we have $ab\in\C_1$, which is a contradiction.  
\end{proof}

The next two lemmas investigate the properties of the vertices of the sets $\mathcal{A}$, and $\mathcal{B}$. 
%
\begin{lemma}\label{lem:help-mega}
Let $G$ be a square geometric ${B}_{a,b}$-graph as in Assumption \ref{ass:nec-condition}.  
\begin{itemize}
\item[(1)] Let $\{ay_1,x_2y_2\}$ and $\{by_1',x_2'y_2'\}$ be rigid pairs. If $ay_2$ and $by_2'$ belong to the same completion $\C_i$, $1\leq i\leq 2$, then $ab\in\C_i$.
\item[(2)] Let $\{ay_1,x_2y_2\}$ and $\{ay_1',x_2'y_2'\}$ be rigid pairs. Then $ay_2$ and $ay_2'$ belong to the same completion $\C_1$ or $\C_2$. Similarly, if $\{by_1,x_2y_2\}$ and $\{by_1',x_2'y_2'\}$ are rigid pairs then $by_2$ and $by_2'$ belong to the same completion $\C_1$ or $\C_2$.
\item[(3)] Fix $i\in\{1,2\}$. If $a<_i X_a\cap X_b<_i b$ and $a<_i Y<_i b$ then all edges of $E(\tilde{G})$ has an end in either $a$-vertices or $b$-vertices .
\end{itemize}
\end{lemma}
\begin{proof}
To prove (1), suppose without loss of generality that $ay_2, by_2'\in\C_1$.
Consider the rigid pair $\{ay_1,x_2y_2\}$. Suppose without loss of generality $a<_1 x_2$. Since $ay_2\in\C_1$ then by (1) of Lemma \ref{lem:x_1<x_2iffy_1<y_2} we have that $y_1<_1 a<_1 x_2$, $y_1<_1 X$ and $Y<_1 x_2$. Now consider the rigid pair $\{by_1',x_2'y_2'\}$. First let $b<_1 x_2'$. Since $a<_1 x_2$ then by (4) of Lemma \ref{lem:completion-properties} we have that $ab\in\C_1$. Now let $x_2'<_1 b$. Since $by_2'\in\C_1$ then by Lemma \ref{lem:x_1<x_2iffy_1<y_2} we have that $x_2'<_1 b<_1 y_1'$. This together with $y_1<_1 a<_1 x_2$, $y_1\in N(y_1')$, $x_2\in N(x_2')$, and (5) of Lemma \ref{lem:completion-properties} implies that $ab\in\C_1$. 

We now prove (2). Let $\{ay_1,x_2y_2\}$ and $\{ay_1',x_2'y_2'\}$ be rigid pairs. Suppose to the contrary that $ay_2\in\C_1$ and $ay_2'\in\C_2$. Then, by (2) of Lemma \ref{lem:x1<x2-iff-x'1<x'2}, for all $i\in\{1,2\}$ we have that $a<_i x_2$ if and only if $x_2'<_i a$. First let $i=1$, and without loss of generality assume that $a<_1 x_2$. Then $x_2'<_1 a<_1 x_2$. Since $x_2, x_2'\in N(b)$ then by (1) of Lemma \ref{lem:completion-properties} we have that $ab\in\C_1$. An analogous discussion for $i=2$ shows that $ab\in\C_2$. This implies that $ab\in\C_1\cap\C_2$, which contradicts  $\C_1\cap\C_2=\emptyset$. 
This proves the first statement of (2). The proof of the second statement follows from an analogous discussion.

Without loss of generality, we prove (3) for $i=1$. Let $a<_1 X_a\cap X_b<_1 b$ and $a<_1 Y<_2 b$. Then, for all $y\in Y$, and all $x\in X_a\cap X_b$ we have that either $a<_1 x<_1 y<_1 b$, or $a<_1 y<_1 x<_1 b$. Moreover, $a, b\in N(x)$, and thus for all $y\in Y$ and all $x\in X_a\cap X_b$ we have $xy\in\C_1$. 
Now suppose to the contrary that $E(\tilde{G})$ has an edge which has no ends in $a$-vertices and $b$-vertices. Then there is a rigid pair $\{x_1y_1,x_2y_2\}$ in $G$ and $x_1y_2\sim^* x_2y_1$. By Proposition \ref{prop:completion-empty}, we know that $x_1y_2$ and $x_2y_1$ belong to different completions. But we know that $x_1y_2, x_2y_1\in\C_1$. This finishes the proof.

\end{proof}
\begin{lemma}\label{lem:help-mega-2}
Let $G$ be a square geometric ${B}_{a,b}$-graph as in Assumption \ref{ass:nec-condition}, $a_1, a_2$ be two different vertices of  $X_a\setminus X_b$, and $b\in X_b\setminus X_a$. Let $\{a_1y_1,x_2y_2\}$ and $\{a_2y_1',x_2'y_2'\}$ be rigid pairs. If $a_1y_2$ and $a_2y_2'$ belong to different completions $\C_1$ and $\C_2$, then
\begin{itemize}
\item[(1)] Each completion $\C_1$ and $\C_2$ contains exactly one of the non-edges $a_1b$ and $a_2b$.
\item[(2)] Fix $i\in\{1,2\}$ and suppose that $a_1b\notin\C_i$
\begin{itemize}
\item[(2.1)] If $a_1y_2\notin\C_i$ then for all $x\in X_a\cup X_b$ and all $y\in N_Y(b)\setminus N_Y(x)$, we have $xy\in\C_i$.
\item[(2.2)] If $a_1y_2\in\C_i$ then for all $y\in Y\setminus N_Y(b)$ we have $by\in\C_i$.
\end{itemize}
\end{itemize}
Moreover, if $b_1, b_2$ are different vertices of $X_b\setminus X_a$ and $a\in X_a\setminus X_b$ then the result holds.
\end{lemma}
\begin{proof}
Suppose $\{a_1y_1,x_2y_2\}$ and $\{a_2y_1',x_2'y_2'\}$ are rigid pairs and $a_1y_2$ and $a_2y_2'$ belong to different completions $\C_1$ and $\C_2$. 

We first prove (1). By (2) of Lemma \ref{lem:x1<x2-iff-x'1<x'2}, for all $i\in\{1,2\}$ we have that $a_1<_i x_2$ if and only if $x_2'<_i a_2$. Suppose without loss of generality that $a_1<_i x_2$ and $x_2'<_i a_2$. Then either $a_1<_i a_2$ or $a_2<_i a_1$. 

\begin{itemize}
\item[(i)] Let $a_2<_i a_1$. Then $x_2'<_i a_2<_i a_1<_i x_2$. Since $x_2, x_2'\in N(b)$ then by (1) of Lemma \ref{lem:completion-properties} we have $a_1b, a_2b\in\C_i$.
\item[(ii)] Let $a_1<_i a_2$. If $b<_i a_1$ then $b<_i a_1<_i x_2$. Since $x_2\in N(b)$ then $a_1b\in\C_i$. If $a_2<_i b$ then $x_2'<_i a_2<_i b$, and thus $a_2b\in\C_i$. Moreover, if $a_1<_i b<_i a_2$ then $a_1b, a_2b\in\C_i$.
\end{itemize}

This implies that each completion $\C_1$ and $\C_2$ contains at least one of $a_1b$ and $a_2b$. Since $\C_1\cap\C_2=\emptyset$ then either $a_1b\in\C_2\setminus\C_1$ and $a_2b\in\C_1\setminus\C_2$ or $a_1b\in\C_1\setminus\C_2$ and $a_2b\in\C_2\setminus\C_1$. Therefore, case (i) and case (ii) when $a_1<_i b<_i a_2$ cannot occur.

We now prove (2) for $i=1$. The proof for $i=2$ is analogous. First let $a_1y_2\in\C_1$ and $a_1b\notin\C_1$. Consider the rigid pair $\{a_1y_1,x_2y_2\}$. Suppose without loss of generality that $a_1<_1 x_2$. Since $a_1y_2\in\C_1$ then by (2) of Lemma \ref{lem:x_1<x_2iffy_1<y_2} we know that $y_1<_1 a_1<_1 x_2$. Since $x_2\in N(b)$ and $a_1b\notin\C_1$ then by (2) of Lemma \ref{lem:completion-properties} we have $a_1<_1 b$. Then by (2) of Lemma \ref{lem:location-1-a,b-rigid-free-conditions} we have $a_1<_1 X_a\cap X_b<_1 b$. Moreover, since $y_1<_1 a_1$ then by (2) of Lemma \ref{lem:location-2-a,b-rigid-free-conditions}, we know that $Y<_1 b$, and thus $y_1<_1 a_1<_1 N_Y(b)<_1 b$.

For all $x\in X_a\cap X_b$ and all $y\in N_Y(b)\setminus N_Y(x)$, we have that $a_1<_1 x<_1 b$ and $a_1<_1 y<_1 b$. Therefore, either $x<_1 y<_1 b$ or $a_1<_1 y<_1 x$. Since $a_1, b\in N(x)$, we have that $xy\in\C_1$. 

We know that $y_1<_1 a_1<_1 N_Y(b)$, and thus by (3) of Lemma \ref{lem:completion-properties}, for all $y\in Y\setminus N_Y(a_1)$, we have that $a_1y\in\C_1$. In particular, for all $y\in N_Y(b)\setminus N_Y(a_1)$, $a_1y\in\C_1$.

Consider the rigid pairs $\{a_1y_1,x_2y_2\}$ and $\{a_2y_1',x_2'y_2'\}$. Since $a_1y_2$ and $a_2y_2'$ belong to different completions and $a_1<_1 x_2$ then by (2) of Lemma \ref{lem:x1<x2-iff-x'1<x'2} we have that $x_2'<_1 a_2$. This together with $a_1<_1 X_a\cap X_b<_1 b$ implies that $a_1<_1 x_2'<_1 a_2$. Since $a_2\in N(a_1)$ and $a_1b\notin\C_1$ then $a_1<_1 x_2'<_1 a_2<_1 b$. Moreover, $a_1<_1 N_Y(b)<_1 a_2$. Therefore, for all $y\in N_Y(b)\setminus N_Y(a_2)$ we have that either $x_2'<_1 y<_1 b$ or $a_1<_1 y<_1 x_2'<_1 a_2$. If the latter occurs then $a_2y\in \C_1$. If the former occurs then either $x_2'<_1 y<_1 a_2<_1 b$ or $x_2'<_1 a_2<_1 y<_1 b$. Since $x_2'\in N(b)$ then $a_2y\in\C_1$.    
This implies that, for all $x\in X_a\cup X_b$ and all $y\in N_Y(b)\setminus N_Y(x)$, we have that $xy\in\C_1$.

We now prove (2.2). If $a_1y_2\notin\C_i$ and $a_1b\notin\C_i$ then for all $x\in Xa\cup X_b$ and all $y\in N_Y(b)\setminus N_Y(x)$ we have $xy\in\C_i$.

Let $a_1y_2\notin\C_1$ and $a_1b\notin\C_1$. Consider the rigid pair $\{a_1y_1,x_2y_2\}$. Suppose without loss of generality that $a_1<_1 x_2$. Since $a_1y_2\notin\C_1$ then by (1) of Lemma \ref{lem:x_1<x_2iffy_1<y_2} we know that $a_1<_1 x_2<_1 y_2$ and $a_1<_1 Y$. This together with (1) of Lemma \ref{lem:location-1-a,b-rigid-free-conditions} implies that $N_Y(a_1)<_1 b$.

Now consider the rigid pairs $\{a_1y_1,x_2y_2\}$ and $\{a_2y_1',x_2'y_2'\}$. Since $a_1y_2$ and $a_2y_2'$ belong to different completions and $a_1<_1 x_2$ then by (2) of Lemma \ref{lem:x1<x2-iff-x'1<x'2} we have that $x_2'<_1 a_2$. This together with $a_1<_1 X_a\cap X_b<_1 b$ implies that $a_1<_1 x_2'<_1 a_2$. Moreover, since $a_1y_2\notin\C_1$ then $a_2y_2'\in\C_1$. Therefore, by (1) of Lemma \ref{lem:x1<x2-iff-x'1<x'2} we have that either $x_2'<_1 y_2'<_1 a_2<_1 y_1'$ or $x_2'<_1 a_2<_1 y_2'<_1 y_1'$. If $y_1'<_1 b$ then since $x_2'\in N(b)$ we have $x_2'y_1, a_2y_2'\in\C_1$ which contradicts Proposition \ref{prop:completion-empty}. This implies that $b<_1 y_1'$, and thus $N_Y(a_1)<_1 b<_1 y_1'$. Therefore, by (3) of Lemma \ref{lem:completion-properties} for all $y\in Y\setminus N_Y(b)$ we have that $by\in\C_1$.
\end{proof}
\subsection{Necessity of Condition (1) of Theorem \ref{thm:chp3-main-result}}
\label{subsec:chp3-necessity-coloring}
In this subsection, we will prove that if $G$ is a square geometric graph, as given in Assumption \ref{ass:1}, then there exits a proper 2-coloring of $\tilde{G}$ such that all vertices of $\mathcal{A}$ are red and all vertices of $\mathcal{B}$ are blue. 
\begin{lemma}\label{lem:mega-lemma}
Let $G$ be a square geometric $B_{a,b}$-graph as in Assumption \ref{ass:nec-condition}. Then
either $\mathcal{A}\subseteq\C_1$ and $\mathcal{B}\subseteq\C_2$, or $\mathcal{A}\subseteq \C_2$ and $\mathcal{B}\subseteq\C_1$.
\end{lemma}
\begin{proof}
Let $G$ be as in Assumption \ref{ass:nec-condition}. We first prove (1). Let $a\in X_a\setminus X_b$ and $ay, ay'\in \mathcal{A}$. Then there are rigid pairs $\{ay_1, x_2y_2\}$ and $\{ay_1',x_2'y_2'\}$. By (2) of Lemma \ref{lem:help-mega}, we know that the $a$-vertices, $ay_2$ and $ay_2'$, belong to the same completions. This implies that all the $a$-vertices of $\mathcal{A}$ belong to the same completion. An analogous discussion and (2) of Lemma \ref{lem:help-mega} prove that all the $a_2$-vertices of $\mathcal{A}$ belong to the same completion, and all the $b_1$-vertices of $\mathcal{B}$ belong to the same completion. Now consider the following cases:

\vspace{0.2cm}

\vspace{0.2cm}
\noindent{\bf Case 1.} $\mathcal{A}\neq\emptyset$ and $\mathcal{B}=\emptyset$, or $\mathcal{A}=\emptyset$ and $\mathcal{B}\neq\emptyset$: We only discuss the former case. The proof of  the case $\mathcal{A}=\emptyset$ and $\mathcal{B}\neq\emptyset$ follows by symmetry of $X_a$ and $X_b$. Let $a_1, a_2$ be two distinct vertices of $X_a\setminus X_b$. By the above discussion, we know that all the $a_1$-vertices in $\mathcal{A}$ belong to the same completion, and all the $a_2$-vertices in $\mathcal{A}$ belong to the same completion. We now prove that all the $a_2$-vertices and $a_1$-vertices of $\mathcal{A}$ belong to the same completion.
So suppose that $a_1y_2, a_2y_2'\in\mathcal{A}$. Then there are rigid pairs $\{a_1y_1,x_2y_2\}$ and $\{a_2y_1',x_2'y_2'\}$. Suppose to the contrary that $a_1y_2$ and $a_2y_2'$ belong to different completions. Without loss of generality let $a_1y_2\in\C_1$ and $a_2y_2'\in\C_2$. Then by (1) of Lemma \ref{lem:help-mega-2} we know that, for $b\in X_b\setminus X_a$, either $a_1b\in\C_1\setminus\C_2$ and $a_2b\in\C_2\setminus\C_1$ or $a_1b\in\C_2\setminus\C_1$ and $a_2b\in\C_1\setminus\C_2$. 

First let $a_1b\in\C_1\setminus\C_2$ and $a_2b\in\C_2\setminus\C_1$. Since $a_1b\in\C_1$ and $a_1y_2\in\C_1$ then, by (2.1) of Lemma \ref{lem:help-mega-2}, for all $y\in Y\setminus N_Y(b)$, we have that $by\in\C_1$. Moreover, $a_2b\in\C_2$ and $a_2y_2'\in\C_2$, and thus by (2.1) of Lemma \ref{lem:help-mega-2} for all $y\in Y\setminus N_Y(b)$ we have that $by\in\C_2$. This implies that, for all $y\in Y\setminus N_Y(b)$, we have that $by\in\C_1\cap\C_2$. Since $\C_1\cap\C_2=\emptyset$ then $Y\setminus N_Y(b)=\emptyset$, which implies that $N_Y(b)=Y$. This case is not part of Assumption \ref{ass:nec-condition}.

Now let $a_1b\in\C_2\setminus\C_1$ and $a_2b\in\C_1\setminus\C_2$. Since $a_1b\in\C_2$ and $a_1y_2\notin\C_2$ then by (2.2) of Lemma \ref{lem:help-mega-2} for all $x\in X_a\cup X_b$ and all $y\in N_Y(b)\setminus N_Y(x)$ we have that $xy\in\C_2$.
Moreover, $a_2b\in\C_1$ and $a_2y_2'\notin\C_1$, and thus by (2.2) of Lemma \ref{lem:help-mega-2} for all $x\in X_a\cup X_b$ and all $y\in N_Y(b)\setminus N_Y(x)$ we have that $xy\in\C_2$. This implies that for all $x\in X_a\cup X_b$ and all $y\in N_Y(b)\setminus N_Y(x)$ we have that $xy\in\C_1\cap\C_2$. Since $\C_1\cap\C_2=\emptyset$ then, for all $x\in X_a\cup X_b$, we have that $N_Y(b)\setminus N_Y(x)=\emptyset$, which implies that for all $x\in X_a\cup X_b$, $N_Y(b)\subseteq N_Y(x)$. This case is not part of Assumption \ref{ass:nec-condition}. Therefore, for all $a\in X_a\setminus X_b$, we have that all vertices of $\mathcal{A}$ belong to the same completion.

\vspace{0.2cm}
\noindent {\bf Case 3.} $\mathcal{A}\neq\emptyset$ and  $\mathcal{B}\neq\emptyset$: First recall that, as we discussed in the very beginning of the proof, for all $a\in X_a\setminus X_b$ and all $b\in X_b\setminus X_a$, all the $a$-vertices  of $\mathcal{A}$ belong to the same completion, and all the $b$-vertices of $\mathcal{B}$ belong to the same completion. We now prove that an $a$-vertex of $\mathcal{A}$, and a $b$-vertex of $\mathcal{B}$ belong to different completions. 

Suppose that $ay\in\mathcal{A}$, and $by'\in\mathcal{B}$.  
Then there are rigid pairs $\{ay_1,xy\}$ and $\{by_2,x'y'\}$ in $G$ with $x,x'\in X_a\cap X_b$. By contradiction let $ay\in\C_1$ and $by'\in\C_1$. 
By (1) of Lemma \ref{lem:help-mega} we know that $ab\in\C_1$, and thus by Proposition \ref{prop:completion-empty} we have that $ab\notin\C_2$. Suppose without loss of generality that $a<_2 b$. Then by (2) of Lemma \ref{lem:location-1-a,b-rigid-free-conditions} we have that $a<_2 X_a\cap X_b<_2 b$, $a<_2 N_Y(b)$, and $N_Y(a)<_2 b$. Moreover, by (2) of Lemma \ref{lem:location-2-a,b-rigid-free-conditions} either $a<_2 Y$ or $Y<_2 b$. Without loss of generality let $a<_2 Y$. First suppose that we also have $Y<_2 b$. This implies that $a<_2 X_a\cap X_b<_2 b$ and $a<_2 Y<_2 b$. Then by (3) of Lemma \ref{lem:help-mega} we know that all edges of $E(\tilde{G})$ has either an end in $a$-vetices or an end in $b$-vertices. This case is excluded in Assumption \ref{ass:1}. 

Therefore, $Y<_2 b$ cannot occur, and thus there is a vertex $y_3\in Y$ such that $b<_2 y_3$. We also have $N_Y(a)<_2 b$. Then $N_Y(a)<_2 b<_2 y_3$. Since $y'\in Y\setminus N_Y(b)$ then, by (3) of Lemma \ref{lem:completion-properties}, we have that $by'\in\C_2$, which contradicts our assumption ($by'\in\C_1$).
This proves that either $\mathcal{A}\subseteq\C_1$ and $\mathcal{B}\subseteq\C_2$, or $\mathcal{A}\subseteq\C_2$ and $\mathcal{B}\subseteq\C_1$.
\end{proof}
The necessity of Condition (1) of Theorem \ref{thm:chp3-main-result} is a direct result of Lemma \ref{lem:mega-lemma}.
%
%
\subsection{Necessity of Condition (2) of Theorem \ref{thm:chp3-main-result}}
\label{subsec:chap3:Nec-rigid-free}
We devote this subsection to the proof of necessity of Condition (2) of Theorem \ref{thm:chp3-main-result}. As we mentioned in the beginning of the section, the rigid-free conditions exclude specific structures of neighborhoods of vertices of $a\in X_a\setminus X_b$ and $b\in X_b\setminus X_a$. We prove in this subsection that the occurrence of these structures does not allow the existence of linear orders $<_1$ and $<_2$ as in Equation \ref{eq:2-dim-geo}. Indeed, if such structures occur then for every pair of linear orders $<_1$ and $<_2$ at least one of the non-edges of form $ab$ belongs to $\C_1\cap\C_2$, where $\C_1$ and $\C_2$ are completions corresponding to $<_1$ and $<_2$, respectively. 

We prove the necessity of Condition (2) of Theorem \ref{thm:chp3-main-result} in stages through a number of lemmas. These lemmas collect information on non-edges of form $ab$, and the location of the vertices  $a\in X_a\setminus X_b$, and $b\in X_b\setminus X_a$ in linear orders $<_1$ and $<_2$. The main tools of all proofs of this subsection are Lemma \ref{lem:completion-properties} and Definition \ref{def:completion} (definition of completion). 
%

%
%
\begin{lemma}\label{lem:help1-rigid-type1}
Suppose that $G$ is a square geometric $B_{a,b}$-graph.
\begin{itemize}
\item[(1)] Let $\{ay_1,x_2y_2\}$ be a rigid pair with $y_1, y_2\in N_Y(b)$.
Then $ay_2$ and $ab$ belong to the same completion. 
%
\item[(2)] Let $\{by_1',x_2'y_2'\}$ be a rigid pair with $y_1',y_2'\in N_Y(a)$.
Then $by_2'$ and $ab$ belong to the same completion.
\end{itemize}
\end{lemma}
\begin{proof}
We only prove (1). The proof of (2) is analogous. Suppose $\{ay_1,x_2y_2\}$ is a rigid pair with $y_1, y_2\in N_Y(b)$ and $ay_2\in\C_1$. First suppose $a<_1 x_2$. Then by (2) of Lemma \ref{lem:x_1<x_2iffy_1<y_2} we have either  $y_1<_1 y_2 <_1 a<_1 x_2$ or $y_1<_1 a<_1 y_2 <_1 x_2$. In both cases, $y_1<_1 a<_1 x_2$. Since $x_2, y_1\in N(b)$, by part (1) of Lemma \ref{lem:completion-properties} we have $ab\in\C_1$. 

Now suppose that $x_2<_1 a$. Then, by Lemma \ref{lem:x_1<x_2iffy_1<y_2}, we have that either $ x_2<_1 a<_1 y_2<_1 y_1 $ or $x_2<_1 y_2<_1 a<_1 y_1$. Thus $x_2<_1 a<_1 y_1$. Again since $x_2, y_1\in N(b)$, then by part (1) of Lemma \ref{lem:completion-properties}, we have that $ab\in\C_1$. 
If $ay_2\in\C_2$, then the same argument with $<_1$ replaced by $<_2$ shows $ab\in\C_2$. This finishes the proof.
\end{proof}

In what follows we prove that rigid-free conditions of Definition \ref{def:rigid-free-conditions} hold. Indeed, we prove that  structures (1) and (2) of Figure \ref{fig:all}, and structures (3) and (4)can not occur simultaneously. {\sl I.e.} if one of structures (1) and (2) occurs then non of structures (3) and (4) can occur and vice versa.
\begin{figure}[H]
\centerline
{\includegraphics[width=0.70\textwidth]{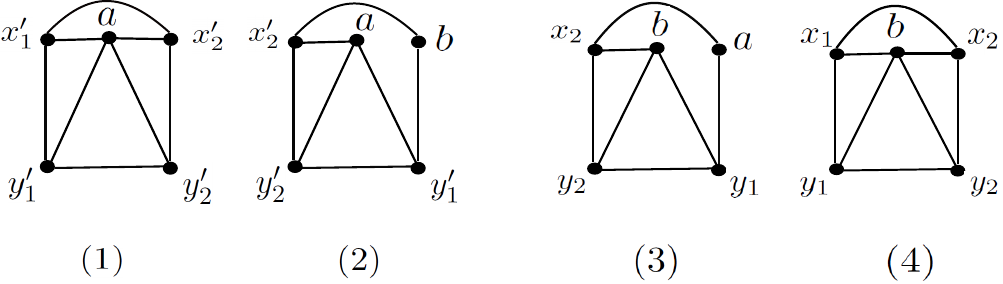}}
\caption{The structures (1), (2) and (3), (4) can not occur simultaneously.\label{fig:all}}
\end{figure} 

\begin{corollary}
Suppose that $G$ is a square geometric graph as in Assumption \ref{ass:nec-condition}. If there is a rigid pair $\{ay_1,x_2y_2\}$  with $y_1, y_2\in N_Y(b)$ then there is no rigid pair $\{by_1',x_2'y_2'\}$ with $y_1',y_2'\in N_Y(a)$. Equivalently, if there is a rigid pair $\{by_1',x_2'y_2'\}$ with $y_1',y_2'\in N_Y(a)$ then there is no rigid pair $\{ay_1,x_2y_2\}$  with $y_1, y_2\in N_Y(b)$.
\end{corollary}
\begin{proof}
Suppose by contradiction that there are rigid pairs $\{ay_1,x_2y_2\}$  with $y_1, y_2\in N_Y(b)$ and $\{by_1',x_2'y_2'\}$ with $y_1',y_2'\in N_Y(a)$. Then $ay_2\sim^* x_2y_1$ and $by_2'\sim^* x_2'y_1'$, and thus $ay_2\in\mathcal{A}$ and $by_2'\in\mathcal{B}$. By part (1) of Lemma \ref{lem:mega-lemma}, we know that $ay_2$ and $by_2'$ belong to different completions $\C_1$ and $\C_2$. By Lemma \ref{lem:help1-rigid-type1}, this implies that $ab$ belongs to both completions $\C_1$ and $\C_2$.
%
This contradicts $\C_1\cap\C_2=\emptyset$. 
\end{proof}

We now prove that structures (1) and (4), as shown in Figure \ref{fig:all}, do not occur simultaneously.
%
%
%
\begin{lemma}\label{lem:help2-rigid-type-1}
Let $\{x_1y_1,x_2y_2\}$ be a rigid pair with $x_1, x_2\in X_a\cap X_b$. For $v\in X_a\cup X_b$, if $y_1, y_2\in N_Y(v)$ then, for all $i\in \{1,2\}$, either $x_1<_i v<_i y_2$ or $y_1<_i v<_i x_2$. Moreover, if $x_1<_i v<_i y_2$ then $x_1<_i Y$ and $X_a\cap X_b <_i y_2$, and if $y_1<_i v<_i x_2$ then $y_1<_i X_a\cap X_b$ and $Y<_i x_2$.
\end{lemma}
\begin{proof}
Assume without loss of generality that $i=1$. We know that $x_1y_2$ and $x_2y_1$ belong to different completions. By symmetry of $x_1y_2$ and $x_2y_1$, we may assume that $x_2y_1\in\C_1$ and $x_1y_2\in\C_2$. Since $v\in X_a\cup X_b$ and $x_1\in X_a\cap X_b$ then we have that $v\in N(x_1)$. Moreover $v\in N(y_2)$. First let $x_1<_1 x_2$. Then by (1) of Lemma \ref{lem:x_1<x_2iffy_1<y_2}, we have that $x_1<_1 v<_1 y_2$, 
and $x_1<_1 Y$. Also, since $X_a\cap X_b\subseteq N(x_1)$, by Lemma \ref{lem:x_1<x_2iffy_1<y_2}, we have that $X_a\cap X_b<_1 y_2$.

If $x_2<_1 x_1$ then a similar argument, and (2) of Lemma \ref{lem:x_1<x_2iffy_1<y_2}, proves that $y_1<_1 v<_1 x_2$, $y_1<_1 X_a\cap X_b$ and $Y<_1 x_2$. 
\end{proof}
\begin{proposition}\label{prop:help-rigid-2}
Suppose $G$ is a square geometric graph as in Assumption \ref{ass:nec-condition}. If there is a rigid pair $\{x_1y_1,x_2y_2\}$  with $y_1, y_2\in N_Y(b)$ and $x_1, x_2\in X_a\cap X_b$ then there is no rigid pair $\{x_1'y_1',x_2'y_2'\}$ with $y_1',y_2'\in N_Y(a)$ and $x_1', x_2'\in X_a\cap X_b$. 
\end{proposition}
\begin{proof}
Suppose that there is a rigid pair $\{x_1y_1,x_2y_2\}$ with $y_1, y_2\in N_Y(b)$ and $x_1, x_2\in X_a\cap X_b$. 
%
By contradiction suppose that there is a rigid pair $\{x_1'y_1',x_2'y_2'\}$ with $y_1',y_2'\in N_Y(a)$ and $x_1', x_2'\in X_a\cap X_b$. Then, by Lemma \ref{lem:help2-rigid-type-1},
either $x_1'<_1 a<_1 y_2'$ or $y_1'<_1 a<_1 x_2'$. Also, either $x_1<_1 b<_1 y_2$ or $y_1<_1 b<_1 x_2$. Thus we have the following cases.
\begin{itemize}
\item Let $x_1<_1 b<_1 y_2$ and $x_1'<_1 a<_1 y_2'$. Since $x_1\in N(a)$ and $x_1'\in N(b)$, by (4) of Lemma \ref{lem:completion-properties}, we have that $ab\in\C_1$. Similarly, if $y_1<_1 b<_1 x_2$ and $y_1'<_1 a<_1 x_2'$ then, by (4) of Lemma \ref{lem:completion-properties}, we have that $ab\in\C_1$.
\item Let $x_1<_1 b<_1 y_2$ and $y_1'<_1 a<_1 x_2'$. Since $x_1\in N(x_2')$ and $y_1'\in N(y_2)$, by (5) of Lemma \ref{lem:completion-properties}, we have that $ab\in\C_1$. Similarly, if $y_1<_1 b<_1 x_2$ and $x_1'<_1 a<_1 y_2'$, by (5) of Lemma \ref{lem:completion-properties}, we have that $ab\in\C_1$.
\end{itemize}
Therefore, for all cases, $ab\in\C_1$. A similar argument for $i=2$ proves that $ab\in\C_2$, and thus $ab\in\C_1\cap\C_2$,
 which contradicts Proposition \ref{prop:completion-empty}.
\end{proof}
We prove in the next Proposition that structures (2) and (4), as shown in Figure \ref{fig:all}, do not occur simultaneously.
%
%
%
\begin{proposition}\label{prop:help-rigid-3}
Suppose that $G$ is a square geometric graph as in Assumption \ref{ass:nec-condition}. If there is a rigid pair $\{x_1y_1,x_2y_2\}$  with $y_1, y_2\in N_Y(b)$ and $x_1, x_2\in X_a\cap X_b$ then there is no rigid pair $\{by_1',x_2'y_2'\}$ with $y_1',y_2'\in N_Y(a)$ and $x_2'\in X_a\cap X_b$. Similarly, if there is a rigid pair $\{x_1'y_1',x_2'y_2'\}$ with $y_1',y_2'\in N_Y(a)$ and $x_1', x_2'\in X_a\cap X_b$ then there is no rigid pair $\{ay_1,x_2y_2\}$  with $y_1, y_2\in N_Y(b)$ and $x_2\in X_a\cap X_b$.
\end{proposition}
\begin{proof}
Suppose to the contrary that
there is a rigid pair $\{x_1y_1,x_2y_2\}$  with $y_1, y_2\in N_Y(b)$ and $x_1, x_2\in X_a\cap X_b$
%
, and there is a rigid pair $\{by_1',x_2'y_2'\}$ with $y_1',y_2'\in N_Y(a)$ and $x_1', x_2'\in X_a\cap X_b$. By Proposition \ref{prop:completion-empty}, we know that the chords $by_2'$ and $x_2'y_1'$ belong to different completions. Without loss of generality, assume that $by_2'\in\C_2$ and $x_2'y_1'\in\C_1$. Then, by Lemma \ref{lem:help1-rigid-type1}, we have that $ab\in\C_2$.
%

By Lemma \ref{lem:help2-rigid-type-1}, we know that either $x_1<_1 b<_1 y_2$ or $y_1<_1 b<_1 x_2$. Consider the following cases.
\begin{itemize}
\item Let $x_1<_1 b<_1 y_2$. Since $\{by_1',x_2'y_2'\}$ is a rigid pair with $by_2'\in\C_2$ then $x_2'y_1\in\C_1$. Then, by Lemma \ref{lem:x_1<x_2iffy_1<y_2}, we know that either $y_2'<_1 x_2'<_1 b$ or $b<_1 x_2'<_1 y_2'$. First let $y_2'<_1 x_2'<_1 b$. Moreover $b<_1 y_2$, and thus $y_2'<_1 b<_1 y_2$. Since $y_2'\in N(y_2)$, by definition of completion, we have that $by_2'\in\C_1$. This contradicts our assumption, $by_2'\in\C_2$.
Now let $b<_1 x_2'<_1 y_2'$. Then we have that $x_1<_1 b<_1 x_2'$. Since $x_1, x_2'\in N(a)$, by (1) of Lemma \ref{lem:completion-properties}, we have that $ab\in\C_1$. 
%
\item Let $y_1<_1 b<_1 x_2$. Then a similar argument to the case $x_1<_1 b<_1 y_2$ proves that $ab\in\C_1$.
\end{itemize}
Therefore, for all cases we have that $ab\in\C_1$. But we already know that $ab\in\C_2$. This implies that $ab\in\C_1\cap\C_2$,
which contradicts $\C_1\cap\C_2=\emptyset$ (Proposition \ref{prop:completion-empty}).
 \end{proof}
By symmetry of $X_a$ and $X_b$ in type-1 graphs, an analogous discussion to the proof of Proposition \ref{prop:help-rigid-3} proves the following proposition.
\begin{proposition}\label{prop:help-rigid-4}
Suppose that $G$ is a square geometric graph as in Assumption \ref{ass:nec-condition}. If there is a rigid pair $\{x_1'y_1',x_2'y_2'\}$ with $y_1',y_2'\in N_Y(a)$ and $x_1', x_2'\in X_a\cap X_b$ then there is no rigid pair $\{ay_1,x_2y_2\}$  with $y_1, y_2\in N_Y(b)$ and $x_2\in X_a\cap X_b$.
\end{proposition}

Now we have all the required results to prove the necessity of rigid-free conditions (Definition \ref{def:rigid-free-conditions}) for a square geometric $B_{a,b}$-graph of Assumption \ref{ass:nec-condition}.
\begin{corollary}
Suppose that $G$ is a $B_{a,b}$-graph as in Assumption \ref{ass:nec-condition}. Then the rigid free conditions of Definition \ref{def:rigid-free-conditions} hold.
\end{corollary}
\section{Sufficient Conditions}
\label{sec:chapterp3-sufficiency}
In this section we present sufficient conditions for a $B_{a,b}$-graph with connected chord graph to be square geometric. If a graph $G$ satisfy the sufficient conditions then we construct two linear orders $<_1$ and $<_2$ for $G$ which satisfy Equation \ref{eq:2-dim-geo}. This proves that $G$ is square geometric. To define the orderings $<_1$ and $<_2$ we need the following auxiliary orderings. First recall from Section \ref{sec:chp3:main-algorithm} that a necessary condition for $G$ to be square geometric is that its chord graph $\tilde{G}$ is bipartite. 
%
%
%
 %
\begin{definition}\label{def-proper-biordering-binate}
Let $G$ be a $B_{a,b}$-graph with bipartite $\tilde{G}$. Consider a proper $2$-coloring of $\tilde{G}$, $f:V(\tilde{G})\rightarrow \{\mbox{red}, \mbox{blue}\}$. The relations associated with the coloring $f$, $<_X$ and $<_Y$, are defined as follows. 
\begin{itemize}
\item $x<_X x'$ if there is a rigid pair $\{xy,x'y'\}$ such that $f(xy')$ is red and $f(x'y)$ is blue, or if $x=x'$.
\item $y<_Y y'$ if there is a rigid pair $\{xy,x'y'\}$ such that $f(xy')$ is red and $f(x'y)$ is blue, or if $y=y'$.
\end{itemize}
\end{definition}
We are now ready to state the sufficient conditions.
\begin{theorem}\label{thm:suff}
Let $G$ be a $B_{a,b}$-graph as in Assumption \ref{ass:1} which satisfies the following conditions:
\begin{itemize}
\item[(1)] There is a proper 2-coloring of $\tilde{G}$ such that all vertices of $\A$ are colored red and all of $\B$ are colored blue and the orderings $<_X$ and $<_Y$ associated to the 2-coloring are partial orders.
\item[(2)] The vertices of both sets $A=\{a_1,\ldots, a_r\}$, and $B=\{b_1,\ldots, b_s\}$ have nested neighborhoods, and the following statement holds: \lq\lq
Either the vertices of $X_a\setminus X_b$ or the vertices of $X_b\setminus X_b$ satisfy rigid free conditions as in Definition \ref{def:rigid-free-conditions}.''
\end{itemize}
\end{theorem}
First note that, for a $B_{a,b}$-graph which satisfies the rigid-free conditions (Definition \ref{def:rigid-free-conditions}), by symmetry of $X_a$ and $X_b$, we can always assume that the vertices of $X_a\setminus X_b$ satisfy the rigid free conditions. 
Throughout the rest of this section we assume that $G$ is as in the following assumption.
\begin{assumption}\label{ass:sufficiency-cond}
Let $G$ be as in Assumption \ref{ass:1}. Assume that the conditions of Theorem \ref{thm:suff} hold. Moreover, assume that vertices of $X_a\setminus X_b$ satisfy rigid free conditions, and $N_Y(a_1)\subset N_Y(a_2)\subset\ldots\subset N_Y(a_r)$ and $N_Y(b_1)\subset N_Y(b_2)\subset\ldots\subset N_Y(b_s)$. Let $f:V(\tilde{G})\rightarrow\{\mbox{red},\mbox{blue}\}$ be a 2-coloring of $\tilde{G}$ whose corresponding relations $<_X$ and $<_Y$, as in Definition \ref{def-proper-biordering-binate}, are partial orders, and $f(\tilde{u})=\mbox{red}$ for all $\tilde{u}\in\mathcal{A}$ and $f(\tilde{u})=\mbox{blue}$ for all $\tilde{u}\in\mathcal{B}$. 
\end{assumption}
We now collect some immediate properties of the partial orders $<_X$ and $<_Y$ as in Assumption \ref{ass:sufficiency-cond}. 
%
%
%
%
%
The next proposition shows that oredrings $<_X$ and $<_Y$ as in Assumption \ref{ass:sufficiency-cond} are always reflexive and antisymmetric.
\begin{proposition}\label{prop:reflexive-anti-version2}
Let $G$ be a $B_{a,b}$-graph with bipartite $\tilde{G}$. Let $f$ be an arbitrary proper 2-coloring of $\tilde{G}$ with corresponding relations $<_X$ and $<_Y$ as in Definition \ref{def-proper-biordering-binate}. Then the restrictions of $<_X$ to $X_a$ and $X_b$, and $<_Y$ are reflexive and antisymmetric.
\end{proposition}
\begin{proof}
Let $f:V(\tilde{G})\rightarrow \{\mbox{red}, \mbox{blue}\}$ be a  proper $2$-coloring of $\tilde{G}$. It is easy to see that for a rigid pair $(x_1y_1,x_2y_2)$ we have $x_1<_X x_2$ if and only if $y_1<_Y y_2.$ Therefore, it is enough to show that the restrictions of $<_X$ to $X_a$ and $X_b$ are reflexive and antisymmetric. We know, by Definition \ref{def-proper-biordering-binate}, that for any $x\in X$ we have that $x<_X x$. This gives us reflexivity of the restrictions of $<_X$ to $X_a$ and $X_b$.

We now prove the antisymmetry. First suppose that $x_1, x_2\in X_a$ and $x_1$ and $x_2$ are related in $<_X$. Then there is a rigid pair $\{x_1y_1,x_2y_2\}$ with $x_1, x_2\in X_a$. Suppose that there is another rigid pair $\{x_1y_1',x_2y_2'\}$. Then, by definition of the chord graph we have, $x_2y_1\sim^* x_1y_2'\sim^* x_2y_1'$. This implies that $f(x_2y_1)=f(x_2y_1')$, and thus for two distinct vertices $x_1,x_2$, only one of $x_1<_X x_2$ or $x_2<_X x_1$ is true. This implies that if $x_1<_X x_2$ and $x_2<_X x_1$ then $x_1=x_2$, and so the restriction $<_X$ to $X_a$ is antisymmetric. An analogous discussion proves that the restriction $<_X$ to $X_b$ is antisymmetric.
\end{proof}
We will use the relations $<_X$ and $<_Y$ to define our desired linear orders $<_1$ and $<_2$ for the graph $G$, that is, orders that satisfy Equation \ref{eq:2-dim-geo}. So, first, we investigate how vertices of the graph $G$ relate in the relations $<_X$ and $<_Y$. Recall that, according to Definition \ref{def-proper-biordering-binate}, for $x_1, x_2$ both in $X_a$ or both in $X_b$, $x_1<_X x_2$ if and only if there is a rigid pair $\{x_1y_1,x_2y_2\}$ such that in the proper 2-coloring $f$, as in Assumption \ref{ass:sufficiency-cond}, $x_1y_2$ is colored red. Specifically, two vertices $x_1, x_2$ both $X_a$ or both in $X_b$ are related in $<_X$ if they are part of a rigid pair. Otherwise their neighborhoods are nested. Similarly, two vertices $y_1, y_2\in Y$ are related in $Y$ if they are part of a rigid pair. Otherwise, they have nested neighborhoods.
%
The next two propositions list some useful properties of the relations $<_X$ and $<_Y$ of the graph $G$, as given in Assumption \ref{ass:sufficiency-cond}. 
\begin{proposition}\label{prop:partial-Ord-X}
Let $G$, $f$, and $<_X$ be as in Assumption \ref{ass:sufficiency-cond}. Suppose $x\in X_a\cup X_b$. Then
\begin{itemize}
\item[(1)] For any $x_1, x_2\in X_a$ or $x_1, x_2\in X_b$, either $x_1$ and $x_2$ are related in $<_X$, or they have nested neighborhoods.
\item[(2)] For all $a\in X_a\setminus X_b$, either $a<_X x$, or $a$ and $x$ have nested neighborhoods in $Y$. Similarly, either $x<_X b$, or $x$ and $b$ have nested neighborhoods in $Y$.
\item[(3)] Vertices $a$ in $X_a\setminus X_b$ and $b$ in $X_b\setminus X_a$,  are not related in $<_X$. 
\end{itemize}
\end{proposition}
\begin{proof}
Part (1) follows directly from the fact that, for any two vertices $x_1, x_2$ in $X_a$ or $X_b$, either they are part of a rigid pair or they have nested neighborhoods. 

We now prove (2). Let $x\in X_a\cap X_b$. Suppose that $a\in X_a\setminus X_b$, and the neighborhoods of $a$ and $x$ are not nested in $Y$. This implies that there are $y_1, y_2\in Y$ such that $\{ay_1, xy_2\}$ is a rigid pair of $G$. Therefore, $ay_2\sim^* xy_1$, and thus $ay_2\in\mathcal{A}$. We know that the 2-coloring $f$ of Assumption \ref{ass:sufficiency-cond} colors all vertices of $\mathcal{A}$ red. Then $f(ay_2)=\mbox{red}$, and so $a<_X x$. Similarly, if the neighborhoods of $b_1$ and $x$ in $Y$ are not nested then there is a rigid pair $\{by_1', xy_2'\}$. This implies that $by_2'\in\mathcal{B}$, and since $f$ colors all vertices of $\mathcal{B}$ blue then $f(by_2')=\mbox{blue}$. Therefore, $x<_X b$. 

To prove (3), let $a\in X_a\setminus X_b$ and $b\in X_b\setminus X_a$.
Since $a\nsim b$, by the definition of a rigid pair, there exist no $y_1, y_2\in Y$ such that $\{ay_1, by_2\}$ is a rigid pair. This implies that $a$ and $b$ are not related in $<_X$.  
\end{proof}
\begin{proposition}\label{prop:partial-Ord-Y}
Let $G$, $f$, and $<_Y$ be as in Assumption \ref{ass:sufficiency-cond}. Suppose $y_1, y_2\in Y$. 
Then one of the following cases occurs.
\begin{itemize}
\item[(1)] The neighborhoods of $y_1$ and $y_2$ are nested in $G$. 
\item[(2)] There is a rigid pair $\{x_1y_1,x_2y_2\}$ in $G$, and thus $y_1$ and $y_2$ are related in $<_Y$.
\item[(3)] $N_X(y_1)\setminus N_X(y_2)\subseteq X_a\setminus X_b$ and $N_X(y_2)\setminus N_X(y_1)\subseteq X_b\setminus X_a$.
\end{itemize}
\end{proposition}
\begin{proof}
Let $y_1, y_2\in Y$. Suppose that the neighborhoods of $y_1$ and $y_2$ are not nested in $G$. Then there are $x_1, x_2\in X_a\cup X_b$ such that $x_1\in N_X(y_1)\setminus N_X(y_2)$ and $x_2\in N_X(y_2)\setminus N_X(y_1)$.  If $x_1\sim x_2$ then $\{x_1y_1,x_2y_2\}$ forms a rigid pair of $G$, and thus $y_1$ and $y_2$ are related in $<_Y$. If $x_1\nsim x_2$, then $x_1\in X_a\setminus X_b$ and $x_2\in X_b\setminus X_a$. This implies that $N_X(y_1)\setminus N_X(y_2)\subseteq X_a\setminus X_b$ and $N_X(y_2)\setminus N_X(y_1)\subseteq X_b\setminus X_a$, and we are done.
\end{proof}
%
We now define two relations $<_1$ and $<_2$ for a graph $G$ of Assumption \ref{ass:sufficiency-cond}. 
\begin{definition}\label{def:linOrd-type2}
Let $G$ be as in Assumption \ref{ass:sufficiency-cond}. Let $<_X$ and $<_Y$ be as in Assumption \ref{ass:sufficiency-cond}. Define
\noindent{Ordering $<_1$:}
\begin{itemize}
\item[1.1.] $x<_1 x'$ if $x<_X x'$ or $N_Y(x)\subseteq N_Y(x')$ for all $x, x'\in X_a\cap X_b$. 
\item[1.2.] $a_1<_1 \ldots<_1 a_r<_1 x<_1 b_1<_1 \ldots<_1 b_s$ for all $x\in X_a\cap X_b$.
\item[1.3.] $y<_1 y'$ if $y<_Y y'$ or $N_X(y')\subseteq N_X(y)$ for all $y, y'\in Y$.
\item[1.4.] $y<_1 y'$ if $N_X(y)\setminus N_X(y')\subseteq X_a\setminus X_b$ and $N_X(y')\setminus N_X(y)\subseteq X_b\setminus X_a$.
\item[1.5.] $y_a<_1 b_1$ and $b_s <_1 y$ for all $y_a\in N_Y(a)$ and all $y\in Y\setminus N_Y(a)$, where $a\in X_a\setminus X_b$.
\item[1.6.] $x<_1 y$ for all $x\in X_a$ and all $y\in Y$.
\end{itemize}

\noindent{Ordering $<_2$: }
\begin{itemize}
\item[2.1.]$x<_2 x'$ if  $x'<_X x$ or $N_Y(x)\subseteq N_Y(x')$  for all $x, x'\in X_a\cup X_b$.
\item[2.2.]$b<_2 a$ for $a\in X_a\setminus X_b$, $b\in X_b\setminus X_a$ with $N_Y(a)\not\subseteq N_Y(b)$ and $N_Y(b)\not\subseteq N_Y(a)$.
\item[2.3.]$y<_2 y'$ if $y'<_Y y$ or $N_X(y')\subseteq N_X(y)$ in $G$.
\item[2.4.]$y<_2 y'$ if $N_X(y')\setminus N_X(y)\subseteq X_a\setminus X_b$ and $N_X(y)\setminus N_X(y')\subseteq X_b\setminus X_a$.
\item[2.5.]$x<_2 y$ for all $x\in X_a\cup X_b$ and all $y\in Y$.
\end{itemize}
\end{definition}
We now briefly discuss the reasoning behind the Definition of \ref{def:linOrd-type2}. Recall that, if $\C_1$ and $\C_2$ are completions of linear orders $<_1$ and $<_2$, then $<_1$ and $<_2$ satisfy Equation \ref{eq:2-dim-geo} if and only if $\C_1\cap\C_2=\emptyset$ (Proposition \ref{prop:completion-empty}). Also, recall from Section \ref{sec:nec-codistion} that for a $B_{a,b}$-graph there are three types of non-edges of $G$: isolated vertices of $\tilde{G}$, chords of rigid pairs, and $ab$ where $a\in X_a\setminus X_b$ and $b\in X_b\setminus X_a$. 
To maintain $\C_1\cap\C_2=\emptyset$ for relations $<_1$ and $<_2$ of Definition \ref{def:linOrd-type2}, we require that non-edges of these three categories belong to at most one of the completions $\C_1$ and $\C_2$.

As we can see, the definitions of $<_1$ and $<_2$ are symmetric on $Y$. However $<_1$ and $<_2$ are not completely symmetric on $X$. The reason for this difference between $<_1$ and $<_2$ is that we want the non-edges of the form $ab$ not to belong to $\C_1$ (completion of $<_1$).
%
The way in which $<_1$ is defined in Definition \ref{def:linOrd-type2} guarantees that the non-edges of form $ab$ do not belong to $\C_1$. We will prove that other non-edges of $G$ also belong to at most one completion.

In the rest of this section, the goal is to prove that the relations $<_1$ and $<_2$ of Definition \ref{def:linOrd-type2} are linear orders satisfying Equation \ref{eq:2-dim-geo}. 

The following proposition presents some useful properties of the relations $<_1$ and $<_2$ of Definition \ref{def:linOrd-type2}.
\begin{proposition}\label{prop:ab-Y-neighbors}
Let $G$ be as in Assumption \ref{ass:sufficiency-cond} and let $<_1$ and $<_2$ be relations of Definition \ref{def:linOrd-type2}. Let $a\in X_a\setminus X_b$ and $b\in X_b\setminus X_a$. Then
\begin{itemize}
\item[(i)] For all $y_a\in N_Y(a)$ and all $y\in Y\setminus N_Y(a)$ we have $y_a<_1 y$.
\item[(ii)] For all $y_b\in N_Y(b)$ and all $y\in Y\setminus N_Y(b)$ we have $y_b<_2 y$.
\end{itemize}
\end{proposition}
\begin{proof}
We only prove (i). The proof of (ii) is analogous. Let $a\in X_a\setminus X_b$. We know by Proposition \ref{prop:partial-Ord-Y} that for any $y_a\in N_Y(a)$ and any $y\in Y\setminus N_Y(a)$ there are three possible cases:
(1) $y_a$ and $y$ are related in $<_Y$. Then by Proposition \ref{prop:partial-Ord-Y}, we have $y_a<_Y y$. This implies that $y_a<_1 y$.  
(2) $y_a$ and $y$ have nested neighborhoods in $G$. Since $a\in N_X(y_a)\setminus N_X(y)$ then we must have $N_X(y)\subseteq N_X(y_a)$, and thus by 1.4 of Definition \ref{def:linOrd-type2} we have $y_a<_1 y$.
(3) Assume $y$ and $y_a$ are not related in $<_Y$ and do not have nested neighborhoods. Since $a\in N_X(y_a)\setminus N_X(y)$ then, by 3 of Proposition \ref{prop:partial-Ord-Y},
we must have $N_X(y_a)\setminus N_X(y)\subseteq X_a\setminus X_b$ and $N_X(y)\setminus N_X(y_a)\subseteq X_b\setminus X_a$. Therefore, by 1.5 of Definition \ref{def:linOrd-type2}, we have $y_a<_1 y$. 
\end{proof}
\begin{proposition}\label{prop:lin-ord-intersection}
Let $<_1$ and $<_2$ are as in Definition \ref{def:linOrd-type2}. Then $<_1$ and $<_2$ are linear orders on $X_a\cap X_b\cup Y$.
\end{proposition}
\begin{proof}
We only prove that $<_1$ is a linear order on $X_a\cap X_b\cup Y$. The proof for $<_2$ follows by the symmetry of $<_1$ and $<_2$ on $X_a\cap X_b$ and $Y$. By Definition \ref{def:linOrd-type2}, if we prove that $<_1$ is a linear order on $X_a\cap X_b$ and $Y$, then we have that $<_1$ is a linear order on $X_a\cap X_b\cup Y$. The reflexivity and antisymmetry of $<_1$ follows directly from the definition of $<_1$, and the fact that the relations $<_X$, $<_Y$, and $\subset$ are partial orders. We now prove the transitivity.

Let $x_1, x_2, x_3\in X_a\cap X_b$ such that $x_1<_1 x_2$ and $x_2<_1 x_3$. By Definition \ref{def:linOrd-type2}, there are a few possible cases. First suppose that $x_1<_X x_2$ and $x_2<_X x_3$. As $<_X$ is transitive, we have $x_1<_X x_3$, and thus $x_1<_1 x_3$. 
If $N_Y(x_1)\subseteq N_Y(x_2)$ and $N_Y(x_2)\subseteq N_Y(x_3)$, then by transitivity of  subset relation, we have $N_Y(x_1)\subseteq N_Y(x_3)$. This implies that $x_1<_1 x_3$. 
Now let $N_Y(x_1)\subseteq N_Y(x_2)$ and $x_2<_X x_3$. Then there are $y_2, y_3\in Y$ such that $\{x_2y_2, x_3y_3\}$ is a rigid pair. If $N_Y(x_1)\subseteq N_Y(x_3)$ then by definition $x_1<_1 x_3$. 
So assume there is $y_1\in N_Y(x_1)$ such that $y_1\nsim x_3$. Since $N_Y(x_1)\subseteq N_Y(x_2)$, and $\{x_2y_2, x_3y_3\}$ is a rigid pair, we have $x_1\nsim y_3$. Therefore, $\{x_1y_1,x_3y_3\}$ and $\{x_2y_1,x_3y_3\}$ are rigid pairs. We have that $\{x_2y_1,x_3y_3\}$ is a rigid pair and $x_2<_X x_3$. Then, we have $y_1<_Y y_3$. Now we know that $\{x_1y_1,x_3y_3\}$ is a rigid pair and $y_1<_Y y_3$, thus $x_1<_X x_3$.

A similar argument shows that $x_1<_X x_3$, when $x_1<_X x_2$ and $N_Y(x_2)\subseteq N_Y(x_3)$, and thus $x_1<_1 x_3$. This finishes the proof of transitivity of $<_1$. Now suppose that $y_1<_1 y_2$ and $y_2<_1y_3$. If $y_1, y_2$, and $y_3$ are either related in $<_Y$ or they have nested neighborhoods, then a discussion analogous to the proof of transitivity of $<_1$ on $x_a\cap X_b$ shows that $y_1<_1 y_3$. Now suppose that one of the pairs $y_1, y_2$ or $y_2, y_3$ are related in $<_1$ as in (3) of Proposition \ref{prop:partial-Ord-Y}. 
Suppose without loss of generality that $N_X(y_1)\setminus N_X(y_2)\subset X_a\setminus X_b$ and $N_X(y_2)\setminus N_X(y_1)\subset X_b\setminus X_a$. Then either $N_X(y_3)\subset N_X(y_2)$ or $y_2<_Y y_3$. Suppose that $N_X(y_3)\subset N_X(y_2)$ then  one of the following occurs: (1) $N_X(y_3)\subset N_X(y_1)$, (2) $N_X(y_1)\setminus N_X(y_3)\subset X_a\setminus X_b$ and $N_X(y_3)\setminus N_X(y_1)\subset X_b\setminus X_a$, or (3) $y_1<_Y y_3$. This implies that $y_1<_1 y_3$. Now let $y_2<_Y y_3$. There there is a rigid pair $(v_2y_2,v_3y_3)$ with $v_2<_X v_3$, and $v_2\in X_a\cap X_b$. This implies that $(v_2y_1,v_3y_3)$ is a rigid pair and thus $y_1<_Y y_3$. Therefore, $y_1<_1 y_3$. This finishes the proof of transitivity of $<_1$. 
\end{proof}
\begin{remark}\label{rem:help-Lin-2}
Suppose that $G$ is as in Assumption \ref{ass:sufficiency-cond}. Let $<_2$ be the relation as in Definition \ref{def:linOrd-type2}. Let $a\in X_a\setminus X_b$, $b\in X_b\setminus X_a$, and $x\in X_a\cap X_b$. 
\begin{itemize}
\item[(1)] By Proposition \ref{prop:partial-Ord-X}, we know that, for all $x\in X_a\cap X_b$, if $x$ and $a$ are related in $<_X$ then $a<_X x$, and if $x$ and $b$ are related in $<_X$ then $x<_X b$. Therefore, by 2.1 of Definition \ref{def:linOrd-type2}, if $a<_2 x$ then $x<_X a$, and if $x<_2 a$ then $N_Y(x)\subseteq N_Y(a)$. Similarly, if $b<_2 x$ then $b<_X x$, and if $x<_2 b$ then $N_Y(x)\subseteq N_Y(b)$.
\item[(2)] If there exists $x\in X_a\cap X_b$ such that $a<_X x$ and $x<_X b$, then $N_Y(a)\not\subseteq N_Y(b)$. Since $x<_X b$ and $a<_X x$ then there are rigid pairs $\{by_1,xy_2\}$ and $\{ay_1',xy_2'\}$ in $G$. If $N_Y(a)\subseteq N_Y(b)$ then $b\sim y_1'$ and $a\nsim y_2$. Therefore, $\{ay_1',xy_2\}$ and $\{by_1',xy_2\}$ are rigid pairs. This implies that $ay_2\sim^* xy_1'\sim^* by_2$. Therefore, $ay_2\in\mathcal{A}$, $by_2\in\mathcal{B}$, and in any proper 2-coloring of $\tilde{G}$, both $ay_2$ and $by_2$ receive the same color. But we know, by Assumption \ref{ass:sufficiency-cond}, that all the vertices of $\mathcal{A}$ are red and all the vertices of $\mathcal{B}$ are blue. This implies that $N_Y(a)\not\subseteq N_Y(b)$.
\end{itemize}
\end{remark}
%
In the next two lemmas, we prove that the relations $<_1$ and $<_2$ are linear orders on $V(G)$.
\begin{lemma}\label{lem:type-2-<1-linOrd}
Let $G$ be as in Assumption \ref{ass:sufficiency-cond}. Suppose that the relation $<_1$ is as in Definition \ref{def:linOrd-type2}. Then $<_1$ is a linear order on $V(G)$.  
\end{lemma}
\begin{proof}
First note that, by Proposition  \ref{prop:lin-ord-intersection}, the relation $<_1$ is a linear order on $V(G)\setminus\{a_1,\ldots, a_r, b_1,\ldots, b_s\}$. We need to prove that, the relation $<_1$ remains a linear order when the vertices of $X_a\setminus X_b$ and $X_b\setminus X_a$ are considered. It directly follows, by Definition \ref{def-proper-biordering-binate}, that $<_1$ is reflexive on $\{a_1,\ldots, a_r, b_1,\ldots, b_s\}$.
Now suppose that $v, v'\in V(G)$, and at least one of $v$ and $v'$ is in $\{a_1,\ldots, a_r, b_1,\ldots, b_s\}$.  Then $v$ and $v'$ are related by one of 1.2, 1.5, and 1.6 of Definition \ref{def:linOrd-type2}. Moreover, by the definition, $<_1$ is antisymmetric on $v, v'$. 

We now prove that $<_1$ is transitive. Let $v_1, v_2, v_3\in V(G)$ such that $v_1<_1 v_2$ and $v_2<_1 v_3$. 
If $v_1, v_2, v_3\in V(G_a)\setminus\{a_1,\ldots, a_r, b_1,\ldots, b_s\}$, then, by Proposition \ref{prop:lin-ord-intersection}, we know that $v_1<_1 v_3$.
If $v_1, v_2, v_3\in X_a\cup X_b$, then, by 1.2 of Definition \ref{def:linOrd-type2}, $<_1$ is transitive on $v_1, v_2, v_3$, and thus $v_1<_1 v_3$.
Moreover, if $v_1, v_2, v_3\in X_a\cup Y$, then by 1.6 of Definition \ref{def:linOrd-type2}, we know that $<_1$ is transitive on $v_1, v_2, v_3$, and thus $v_1<_1 v_3$. 

So assume that among $v_1, v_2, v_3$, one is in $X_b\setminus X_a$, one is in $Y$, and one is in $X_a\setminus X_b$. By Definition \ref{def:linOrd-type2}, we know that vertices of $X_a\setminus X_b$ are minimum elements of $V(G)\setminus \{a_1,\ldots, a_r\}$ under $<_1$. Therefore, $v_1\in X_a\setminus X_b$, and $v_1<_1 v_3$. 
%
This proves that $<_1$ is transitive, and we are done.
\end{proof}
We now prove that $<_2$ is a linear order on $V(G)$.
\begin{lemma}\label{lem:type-2-<2-linOrd}
Let $G$ be as in Assumption \ref{ass:sufficiency-cond}. Suppose that the relation $<_2$ is as in Definition \ref{def:linOrd-type2}. Then $<_2$ is a linear order on $V(G)$.  
\end{lemma}
\begin{proof}
By Proposition \ref{prop:lin-ord-intersection}, we have that $<_2$ is a linear order on the set of vertices of 
$V(G)\setminus \{a_1,\ldots, a_r, b_1, \ldots, b_s\}$. We need to prove that $<_2$ remains a linear order when the vertices of $X_a\setminus X_b$ and $X_b\setminus X_a$ are considered. It directly follows, by Definition \ref{def-proper-biordering-binate}, that $<_2$ is reflexive on $\{a_1,\ldots, a_r, b_1, \ldots, b_s\}$.
Now suppose that $v, v'\in V(G)$, and at least one of $v$ and $v'$ is in $\{a_1,\ldots, a_r, b_1, \ldots, b_s\}$. Then $v$ and $v'$ are related by one of 2.1, 2.2, and 2.5 of Definition \ref{def:linOrd-type2}. Moreover, by the definition, $<_2$ is antisymmetric on $v, v'$. 

We now prove that, for any triple $v_1, v_2, v_3\in V(G)$, the relation $<_2$ is transitive on $v_1, v_2, v_3.$ Since $<_2$ is a linear order on $V(G)\setminus \{a_1, \ldots, a_r, b_1, \ldots, b_s\}$, we assume that at least one of $v_1, v_2, v_3$ is in $\{a_1, \ldots, a_r, b_1, \ldots, b_s\}$. 
If one of $v_1, v_2,$ and $v_3$ is in $Y$, then by 2.5 of Definition \ref{def:linOrd-type2}, we know that $<_2$ is transitive on $v_1, v_2, v_3$. So suppose that $v_1, v_2, v_3\in X_a\cup X_b$.

If $v_1, v_2, v_3\in X_a$ or $v_1, v_2, v_3\in X_b$, then any pair of vertices $v_1, v_2$, and $v_3$ are either related in $<_X$ or they have nested neighborhoods in $Y$. Therefore, they are related in $<_2$ by 2.1 of Definition \ref{def:linOrd-type2}. A discussion similar to the proof of Proposition \ref{prop:lin-ord-intersection} for $X_a\cup X_b$, shows that $<_2$ is transitive on $v_1, v_2,$ and $v_3$. 
Now let $v_1, v_2\in X_a$ and $v_3\in X_b\setminus X_a$. If none of $v_1$ and $v_2$ is in $X_a\setminus X_b$, then $v_1, v_2, v_3$ are all in $X_b$. So assume that at least one $v_1, v_2, v_3$ is in $X-a\setminus X_b$. This implies that $\{v_1, v_2, v_3\}=\{a, x, b\}$, where $a\in X_a\setminus X_b$, $b\i X_b\setminus X_a$, and $x\in X_a\cap X_b$. We consider the following cases:
\begin{itemize}
\item $a<_2 x$ and $x<_2 b$. Then, by (1) of Remark \ref{rem:help-Lin-2}, we have that $N_Y(a)\subseteq N_Y(x)$, and $N_Y(x)\subseteq N_Y(b)$. This implies that $N_Y(a)\subseteq N_Y(b)$, and thus, by 2.1 of Definition \ref{def:linOrd-type2}, $a<_2 b$.
\item $b<_2 x$ and $x<_2 a$. Then, by (1) of Remark \ref{rem:help-Lin-2}, we have that $x<_X b$ and $a<_X x$. Therefore, by (2) of Remark \ref{rem:help-Lin-2}, $N_Y(a)\not\subseteq N_Y(b)$. Then either $N_Y(b)\subseteq N_Y(a)$ or $N_Y(b)\not\subseteq N_Y(a)$. By 2.1 and 2.2 of Definition, for both cases, we have that $b<_2 a$.
\item $x<_2 a$ and $a<_2 b$. Then by (1) of Remark \ref{rem:help-Lin-2}, $a<_X x$, and by 2.1 of Definition \ref{def:linOrd-type2}, $N_Y(a)\subseteq N_Y(b_1)$. If $x$ and $b$ are related in $<_X$, then $x<_X b$. Then, by (2) of Remark \ref{rem:help-Lin-2}, $N_Y(a)\not\subseteq N_Y(b)$, which is a contradiction. So $N_Y(x)\subseteq N_Y(b)$, and thus $x<_2 b$. If $a<_2 b$ and $b<_2 x$ then an analogous discussion proves that $a<_2 x$.
\item $x<_2 b$ and $b<_2 a$. Then by (1) of Remark \ref{rem:help-Lin-2}, $N_Y(x)\subseteq N_Y(b)$. If $N_Y(a)\subseteq N_Y(x)$ then $N_Y(a)\subseteq N_Y(b)$, and thus $a<_2 b$ which is not true. Therefore, either $N_Y(x)\subseteq N_Y(a)$ or $a<_X x$. In both cases, by 2.1 of Definition \ref{def:linOrd-type2}, we have that $x<_2 a$. If $b<_2 a$ and $a<_2 x$, then an analogous discussion proves that $b<_2 x$.
\end{itemize}
This finishes the proof of transitivity of $<_2$.
\end{proof}
Now that we know that the relations $<_1$ and $<_2$ of Definition \ref{def:linOrd-type2} are linear orders, the next step is to show that linear orders $<_1$ and $<_2$ satisfy Equation (\ref{eq:2-dim-geo}). We assume that $\C_1$ and $\C_2$ are completions of $<_1$ and $<_2$, respectively. We first prove that chords of a rigid pair belong to different completions $\C_1$ and $\C_2$. Then we prove that isolated vertices of $\tilde{G}$ belong to at most one completion $\C_1$ and $\C_2$. Note that we already defined $<_1$ and $<_2$ in a way that the non-edges of form $ab$ do not belong to $\C_1$. Recall that definitions of linear orders $<_1$ and $<_2$ are symmetric on $Y$ and $X_a\cap X_b$. 
The next lemma gives us the required results to prove that chords of a rigid pair belong to different completions $\C_1$ and $\C_2$. This is where the rigid-free conditions show up and help us with the proofs.

\begin{lemma}\label{lem:helpchords-notin-C1C2-type-2}
Let $G$, $<_1$ and $<_2$ be as in Assumption \ref{ass:sufficiency-cond}. Suppose $\{x_1y_1,x_2y_2\}$ is a rigid pair and $x_1,x_2\in X_a\cap X_b$. If $x_1<_1 x_2$ then 
\begin{itemize}
\item[(i)] For all $y\in N_Y(x_1)$, $y<_1 y_2$.
\item[(ii)] For all $x\in N_X(y_2)$, $x_1<_1 x$.
\item[(iii)] For any $x\in X$ and $y\in Y$ with $x<_1 x_1$ and $y_2<_1 y$, we have $x\nsim y$.
\end{itemize}
Similarly if $x_1<_2 x_2$ and we replace $<_1$ by $<_2$ in the statements (i)-(iii) then statements (i), (ii), and (iii) hold.
\end{lemma}
\begin{proof}
We prove the lemma for $<_1$. The proof for $<_2$ follows similarly. Suppose that $\{x_1y_1,x_2y_2\}$, where $x_1, x_2\in X_a\cap X_b$, is a rigid pair with $x_1<_1 x_2$. Then, by 1.1 of Definition \ref{def:linOrd-type2}, we know that $x_1<_X x_2$, and thus $y_1<_Y y_2$. This implies that $x_1<_1 x_2<_1 y_1<_1 y_2$. Suppose 

First we prove (i) by contradiction. Let $w\in N_Y(x_1)$ and $y_2<_1 w$. 
Since $\{x_1y_1,x_2y_2\}$ is a rigid pair, $x_1\notin N_X(y_2)$. We also have $x_1\in N_X(w)$, and thus $N_X(w)\not\subseteq N_X(y_2)$. But by assumption $y_2<_1 w$, and thus $y_2<_Y w$. Since $y_2$ and $w$ are related in $<_Y$ there are $u\in N_X(y_2)\setminus N_X(w)$ and $z\in N_X(w)\setminus N_X(y_2)$ such that $\{uy_2,zw\}$ is a rigid pair. Since $x_1\in N_X(w)\setminus N_X(y_2)$ then $\{uy_2,x_1w\}$ is also a rigid pair. This together with the fact that $\{x_1y_1,x_2y_2\}$ is a rigid pair implies that $x_2y_1\sim^* x_1y_2\sim^* uw$ in $\tilde{G}$.
%
%
Since $uw$ and $x_2y_1$ receive the same color, then $x_1<_X x_2$ if and only if $x_1<_X u$. We know that $x_1<_X x_2$. So $x_1<_X u$, and thus $w<_Y y_2$. This together with Definition \ref{def:linOrd-type2} implies that $w<_1 y_2$, which contradicts our assumption. 
%
%

We now prove (ii) for $<_1$. Let $x\in N_X(y_2)$. If $x\in X_b\setminus X_a$ then by Definition \ref{def:linOrd-type2}, $x_1<_1 x$.  If $x\in X_a\cap X_b$, then an argument similar to part (i) proves that $x_1<_1 x$. Now let $x\in X_a\setminus X_b$. Since $y_2\in N_Y(x)$ and $y_1<_1 y_2$, by Proposition \ref{prop:ab-Y-neighbors}, we know that $y_1\in N_Y(x)$. This implies that the neighborhood of $x$ contains the rigid pair $\{x_1y_1,x_2y_2\}$. But, a graph $G$ of Assumption \ref{ass:sufficiency-cond} satisfies the rigid-free conditions. Therefore, $y_2$ has no neighbor in $X_a\setminus X_b$. Then, for all $x\in N_X(y_2)$, we have that $x_1<_1 x$.

The proof of (ii) for $<_2$ is slightly different. Suppose that $\{x_1y_1,x_2y_2\}$ is a rigid pair with $x_1, x_2\in X_a\cap X_b$, and $x_1<_2 x_2$. Then $x_1<_2 x_2<_2 y_1<_2 y_2$. Let $x\in N_X(y_2)$. If $x\in X_a\cap X_b$, then an argument similar to part (i) proves that $x_1<_2 x$. If $x\in X_b\setminus X_a$, then either $x<_X x_1$ or $N_Y(x)\subset N_Y(x_1)$. Then by Definition \ref{def:linOrd-type2}, we have that $x_1<_2 x$. 

Now let $x\in X_a\setminus X_b$. Since $y_2\notin N_Y(x_1)$, we know that $N_Y(x)\not\subseteq N_Y(x_1)$. Then, by (1) of Remark \ref{rem:help-Lin-2}, we know that $x_1<_2 x$. This proves that for all $x\in N_X(y)$, $x_1<_2 x$. 

We now prove (iii). Assume that there are $v\in X_a\cup X_b$ and $z\in Y$ such that $v<_1 x_1<_1 x_2<_1 y_1<_1 y_2<_1 z$. 
Suppose, by contradiction, that $v\sim z$. By Part (i), we have that $x_1\nsim z$, and so $N_Y(v)\not\subseteq N_Y(x_1)$. Then, by Proposition \ref{prop:partial-Ord-X}, we know that $v$ and $x_1$ are related in $<_X$. 
If $v\in\{a_1,a_2\}$, then, by Proposition \ref{prop:partial-Ord-X}, we know that $v<_X x_1$. Now let $v\in X_a\cap X_b$. Since $v<_1 x_1$, by 1.1 of Definition \ref{def:linOrd-type2}, we have that $v<_X x_1$.
This implies that, there is rigid pair $\{vz, x_1w\}$ with $w\in N_Y(x_1)$. Since $v<_X x_1$ then $z<_Y w$. Then, by 1.3 of Definition \ref{def:linOrd-type2}, we have that $z<_1 w$. This is impossible since by Part (i) for all $y\in N_Y(x_1)$, we have that $y<_1 y_2$. Therefore, $v\nsim z$, and we are done. The proof for $<_2$ follows by an analogous discussion.
\end{proof}
\begin{corollary}\label{cor:chords-notin-C1C2-type2}
Let $G$  be as in Assumption \ref{ass:sufficiency-cond}, and $<_1$ and $<_2$ be as in Definition \ref{def:linOrd-type2}. Suppose $\C_1$ and $\C_2$ are completions corresponding to $<_1$ and $<_2$, respectively. Then each chord of a rigid pair $\{x_1y_1,x_2y_2\}$ with $x_1,x_2\in X_a\cap X_b$ belong to at most one of $\C_1$ and $\C_2$.
\end{corollary}
\begin{proof}
Suppose $\{x_1y_1,x_2y_2\}$ is a rigid pair. Without loss of generality, let $x_1<_1 x_2$. Then, by 1.1 of Definition \ref{def:linOrd-type2}, we have that $x_1<_X x_2$, and thus $y_1<_Y y_2$. Then, by 1.3, and 1.6 of Definition \ref{def:linOrd-type2}, we have that $x_1<_1 x_2<_1 y_1<_1 y_2$. Since $x_1\sim y_1$ then $x_2y_1\in\C_1$. By Lemma \ref{lem:helpchords-notin-C1C2-type-2}, we know that $x_1$ and $y_2$ are not between two adjacent vertices, and thus by definition of completion $x_1y_2\notin\C_1$. An analogous discussion proves that $x_2y_1\notin\C_2$.
\end{proof}

The next lemma proves a similar result for rigid pairs $\{xy,x'y'\}$, for which, exactly one of $x$ or $x'$ belongs to $X_a\setminus X_b$ or $X_b\setminus X_a$. Note that both $x$ and $x'$ cannot belong to $X_a\setminus X_b$ since the neighborhoods of vertices of $X_a\setminus X_b$ are nested. Similarly, both $x$ and $x'$ cannot belong to $X_b\setminus X_a$.  

\begin{lemma}\label{lem:chords-notin-C1C2-typ-2-2}
Let $G$ be as in Assumption \ref{ass:sufficiency-cond}, and $<_1$ and $<_2$ be linear orders of Definition \ref{def:linOrd-type2} with completions $\C_1$ and $\C_2$, respectively. If there exists a rigid pair of $G$ of one of the forms $\{ay_1,x_2y_2\}$ or $\{x_1y_1,by_2\}$, then different chords of the rigid pair belong to different completions $\C_1$ and $\C_2$. 
\end{lemma}
\begin{proof}
Now suppose that there is a rigid pair $\{ay_1,x_2y_2\}$. By Proposition \ref{prop:partial-Ord-X}, we know that $a<_X x_2$, and thus $y_1<_Y y_2$. Then by 1.3 and 1.6 of Definition \ref{def:linOrd-type2}, we have $a<_1 x_2<_1 y_1<_1 y_2$. Since $a\in N(y_2)$ then $x_2y_1\in\C_1$. We prove that $ay_2\notin\C_1$. We know by 1.2 of Definition \ref{def:linOrd-type2} that $a<_1 x$ for all vertices $x\in X_a\cap X_b$. This implies that neither $y_2$ nor any vertex $w\in Y$ with $y_2<_1 w$ has a neighbor $u\in X_a\cap X_b$ with $u<_1 a$. Moreover by (1) of Proposition \ref{prop:ab-Y-neighbors} we know that $a$ has no neighbor $w\in Y$ with $y_2<_1 w$. This proves that $ay_2\notin\C_1$. 

Now suppose that there is a rigid pair $\{x_1y_1,by_2\}$. Then, by Proposition \ref{prop:partial-Ord-X,} we have $x_1<_X b$, and thus $y_1<_Y y_2$. By 2.1, 2.3, and 2.5 of Definition \ref{def:linOrd-type2} we have $b<_2 x_1<_2 y_1<_2 y_2$. Since $b\in N(y_1)$ we have $x_1y_1\in\C_2$. We prove $by_2\notin\C_2$. By Proposition \ref{prop:ab-Y-neighbors}, for all $y_b\in N_Y(b)$, we have that $y_b<_2 y$. Moreover, if $x$ and $b$ are related in $<_X$ then, by Proposition \ref{prop:partial-Ord-X}, $x<_X b$. Then by 2.1 of Definition \ref{def:linOrd-type2}, $b<_2 x$. Therefore, if $x<_2 b$ then we must have $N_Y(x)\subseteq N_Y(b)$. This implies that, for all vertices $x<_2 b$, we have $N_Y(x)<_2 y_2$. Therefore $b$ and $y_2$ are not between two adjacent vertices in $<_2$, and thus $by_2\notin\C_2$. This finishes the proof of the lemma. 
\end{proof}

As we mentioned earlier, to prove that $<_1$ and $<_2$ satisfy Equation \ref{eq:2-dim-geo},
,we need to show that completions $\C_1$ and $\C_2$ have empty intersection.
By Corollary \ref{cor:chords-notin-C1C2-type2} and Lemma \ref{lem:chords-notin-C1C2-typ-2-2}, we know that non-edges which correspond to chords of rigid pairs belong to at most one completion $\C_1$ or $\C_2$. Moreover, recall that non-edges of form $ab$ do not belong to $\C_1$. Therefore, to prove that $\C_1\cap\C_2=\emptyset$ we only need to show that non-edges corresponding to isolated vertices of $\tilde{G}$ belong to at most one of the completion $\C_1$ or $\C_2$.
\begin{lemma}\label{lem:type2-isolated-linOrd1,2}
Let $G$ be as in Assumption \ref{ass:sufficiency-cond}. Suppose $<_1$ and $<_2$ are linear orders as in Definition \ref{def:linOrd-type2}. Suppose $uw$ is an isolated vertex of $\tilde{G}$, and $u\in X_a\cap X_b$. 
%
\begin{itemize}
\item[(i)] For all $y\in N_Y(u)$ we have $y<_2 w$.
\item[(ii)] For all $x\in X_a\cup X_b$ with $x<_2 u$ we have $N_Y(x)\subseteq N_Y(u)$.
\end{itemize}
\end{lemma}
\begin{proof}
First we prove (i). Let $uw$ be an isolated vertex of $\tilde{G}$. By Definition of $\tilde{G}$, we know that $uw$ is not a chord of any rigid pair of $G$. We first prove (i). Suppose by contradiction that $y\in N_Y(u)$ and $w<_2 y$. By Proposition \ref{prop:partial-Ord-Y} there are three possible cases: 
\begin{itemize}
\item[(1)] $N_X(y)\subseteq N_X(w)$. We have $u\in N_X(y)\setminus N_X(w)$. This implies that $N_X(y)\not\subseteq N_X(w)$ in $G$.
\item[(2)] $N_X(w)\setminus N_X(y)\subset X_b\setminus X_a$ and $N_X(y)\setminus N_X(w)\subseteq X_a\setminus X_b$. But $u\in N_X(y)\setminus N_X(w)$ and $u\in X_a\cap X_b$. This contradicts $N_X(y)\setminus N_X(w)\subseteq X_a\setminus X_b$. 
\item[(3)] $y<_Y w$. This implies that there is a rigid pair $\{xw,x'y\}$. Since $u\in N_X(y)\setminus N_X(w)$ then $\{xw,uy\}$ is also a rigid pair with chords $xy$ and $uw$. This contradicts the fact that $uw$ is an isolated vertex of $\tilde{G}$, and thus for all $y\in N_Y(u)$ we have $y<_1 w$.
\end{itemize}
We now prove (ii). We know that $uw$ is not chord of any rigid pair of $G$. Therefore, by definition of rigid pair, for any $x\in X_a\cup X_b$ either $N_Y(x)\subseteq N_Y(u)$ or $N_Y(u)\subseteq N_Y(x)$. If $N_Y(u)\subseteq N_Y(x)$ then by 2.1 of Definition \ref{def:linOrd-type2} $u<_2 x$. But $x<_2 u$, and thus $N_Y(x)\subseteq N_Y(u)$. 
\end{proof}
\begin{corollary}\label{cor:type-2-isolated-linOrd1,2}
Let $G$ and be as in Assumption \ref{ass:sufficiency-cond}, and $<_1$ and $<_2$ be as in Definition \ref{def:linOrd-type2}. Suppose $\C_1$ and $\C_2$ are completions corresponding to $<_1$ and $<_2$, respectively. Then isolated vertices, $uw$, of $\tilde{G}$ with $u\in X_b$ do not belong to $\C_2$. Moreover, isolated vertices of $\tilde{G}$ of form $aw$, with $a\in X_a\setminus X_b$ do not belong to $\C_1$.
\end{corollary}
\begin{proof}
Let $uw$ be an isolated vertex of $\tilde{G}$. If $u\in X_a\cap X_b$, then, by Lemma \ref{lem:type2-isolated-linOrd1,2}, we know that $u$ and $w$ are not in between two adjacent vertices in linear order $<_2$. This implies that $uw\notin\C_2$. Now let $u\in X_b\setminus X_a$, and let $u=b$. By Proposition \ref{prop:ab-Y-neighbors}, we know that, for all $y_b\in N_Y(b)$ and for all $y\in Y\setminus N_Y(b)$,we have that $y_b<_2 y$, and thus $y_b<_2 w$. Moreover, for any $x<_2 b$, we know that $N_Y(x)\subseteq N_Y(b)$. This implies that $b$ and $w$ are not in between two adjacent vertices in $<_2$, and thus $bw\notin\C_2$.

Let $a\in X_a\setminus X_b$, and suppose that $u=a$. We know, by Proposition \ref{prop:ab-Y-neighbors}, that, for all $y_a\in N_Y(a)$ and for all $y\in Y\setminus N_Y(a)$, we have that $y_a<_2 y$, and thus $y_a<_2 w$. Moreover, by 1.2 of Definition \ref{def:linOrd-type2}, we know that for all $x\in X_a\cap X_b$, we have that $a<_1 x$. This implies that $a$ and $w$ are not in between two adjacent vertices in $<_1$, and thus $b_1w\notin\C_1$.
\end{proof}
We now use the obtained results to prove that $<_1$ and $<_2$ as in Definition \ref{def:linOrd-type2} are linear orders satisfying Equation \ref{eq:2-dim-geo}.
\begin{theorem}\label{thm:type-2-linOrd}
Let $G$ be as in Assumption \ref{ass:sufficiency-cond}, and $<_1$ and $<_2$ be as in Definition \ref{def:linOrd-type2}. Then relations $<_1$ and $<_2$ are linear orders which satisfy Equation \ref{eq:2-dim-geo}. 
\end{theorem}
\begin{proof}
Let $G$ be as in Assumption \ref{ass:sufficiency-cond}, and linear orders $<_1$ and $<_2$ be as in Definition \ref{def:linOrd-type2}. By Lemmas \ref{lem:type-2-<1-linOrd} and \ref{lem:type-2-<2-linOrd}, we know that $<_1$ and $<_2$ are linear orders. Let $\C_1$ and $\C_2$ be completions of $<_1$ and $<_2$, respectively. By Corollary \ref{cor:chords-notin-C1C2-type2} and Lemma \ref{lem:chords-notin-C1C2-typ-2-2}, we know that non-edges that are chords of a rigid pair belong to different completions. Moreover, by Corollary \ref{cor:type-2-isolated-linOrd1,2} we know that non-isolated vertices of $\tilde{G}$ belong to at most one completion $\C_1$ or $\C_2$. Also, $ab\notin\C_1$. This implies that $\C_1\cap\C_2=\emptyset$, and we are done.
 \end{proof}
\section{Conditions of Theorems \ref{thm:chp3-main-result} and \ref{thm:suff} can be checked in polynomial-time}
\label{sec:chp3:main-algorithm}

In this section, for a $B_{a,b}$-graph $G$, as given in Assumption \ref{ass:1}, we show that conditions of Theorems \ref{thm:chp3-main-result} and \ref{thm:suff} can be checked in $n^4$ steps, where $n$ is the order of the graph $G$. 
In order to check conditions of Theorems \ref{thm:chp3-main-result} and \ref{thm:suff}, first we should construct the chord graph of a $B_{a,b}$-graph $G$. 

The vertices of $\tilde{G}$ are the non-edges of $G$. A non-edge $x_1y_2$ is adjacent with another non-edge $x_2y_1$ if and only if (1) $x_1\sim x_2$  and (2) $y_1\in N_Y(x_1)$ and $x_2\in N_X(y_2)$. Therefore, to form the chord graph of $G$ first for all $x\in X_a\cup X_b$ and all $y\in Y$ we find $N_Y(x)$ and $N_X(y)$, respectively. Let $x_1y_2$ be a non-edge of $G$. Let $\mathcal{E}$ be the set of all $xy$ such that $x\in N_X(y_2)\cap N(x_1)$ and $y\in N_Y(x_1)$. The neighborhood of $x_1y_2$ in $\tilde{G}$ is consist of all $xy\in \mathcal{E}\cap E(G^c)$. Since the order of $G$ is $n$, we can form $\tilde{G}$ in at most $n^4$ steps.

We now discuss the rigid-free conditions. Recall that a vertex $u\in V(G)$ is rigid-free with respect to $S\subseteq V(G)$ if there is no rigid pair $\{x_1y_1,x_2y_2\}$ such that $x_1, x_2\in S$ and $y_1, y_2\in N_Y(u)$. 

For each non-isolated vertex, $xy$, of $\tilde{G}$ let $y\in R(X_a)$ if $x\in X_a$, and $y\in R(X_b)$ if $x\in X_b$. Now $a\in X_a\setminus X_b$ is rigid-free with respect to $\{x_1,x_2\}$ and $\{x,b\}$, where $x_1, x_2, x\in X_b$ and $b\in X_b\setminus X_a$, if and only if $|N_Y(a)\cap R(X_b)|\leq 1$. Similarly, $b\in X_b\setminus X_a$ is rigid-free with respect to $\{x_1,x_2\}$ and $\{x,a\}$, where $x_1, x_2, x\in X_a$ and $a\in X_a\setminus X_b$, if and only if $|N_Y(b)\cap R(X_a)|\leq 1$. Since we can find $R(X_a)$ and $R(X_b)$ in at most $n^4$ steps, rigid-free conditions can be checked in at most $n^4$ steps.

%
Now we will look into the coloring conditions of Theorems \ref{thm:chp3-main-result} and \ref{thm:suff}. First to check whether $\tilde{G}$ is bipartite we perform a BFS to properly color its vertices with two colors. Since $\tilde{G}$ is connected if its bipartite then it has only one possible 2-coloring. If the process of 2-coloring of $\tilde{G}$ fails then $G$ is not square geometric. Moreover, if the set $\A$ and the set $\B$ do not belong to different color classes or either of $\A$ or $\B$ has vertices in both color classes then again $G$ is not square geometric (Theorem \ref{thm:chp3-main-result}). 

We now discuss the condition of Theorem \ref{thm:suff} which requires the orderings $<_X$ and $<_Y$ associated to the proper 2-coloring of $\tilde{G}$ to be partial orders. We already now that orderings $<_X$ and $<_Y$ are reflexive and antisymmetric. To check the transitivity we perform the following steps. For each vertex $v\in X_a\cup X_b$ we define $Out(v)$ to be the set of vertices $u\in X_a\cup X_b$ such that $v<_X u$. For a vertex $v\in Y$, $Out(v)$ is defined similarly. Then $<_X$ and $<_Y$ are transitive whenever for all $v\in V(G)$ the statement \lq\lq for all $u\in Out(v)$, we have $Out(u)\subset Out(v)$''. Since $|V(G)|=n$ then the transitivity of $<_X$ and $<_Y$ can be checked in at most $n^2$ steps. 
\bibliographystyle{plain}
\bibliography{B-Square-Oct2016.bib}

\end{document}